\documentclass[11pt]{article}%
\usepackage{amssymb}
\usepackage{latexsym}
\usepackage{graphicx}
\usepackage{amscd}
\usepackage{amsmath}
\usepackage{amsfonts, color}
\usepackage{amsthm}
\usepackage{amssymb}
\usepackage[active]{srcltx}
\newtheorem{theorem}{Theorem}[section]
\newtheorem{proposition}[theorem]{Proposition}

\newtheorem{remark}[theorem]{Remark}

\newtheorem{lemma}[theorem]{Lemma}
\newtheorem{definition}[theorem]{Definition}

   \newcommand{\Hhat}[1]{\stackrel{\begin{picture}(60,5)\put(0,0){\line(6,1){30}}\put(30,5){\line(6,-1){30}}\end{picture}}{#1}}
\newcommand{\hd}[1]{{\hat{d}}^}

\begin{document}
\title{Hypercyclicity, existence and approximation results for convolution operators on spaces of entire functions}

\author{Vin\'icius V. F\'avaro\thanks{The first named author is supported by FAPESP Grant 2014/50536-7; FAPEMIG Grant PPM-00086-14; and CNPq Grants 482515/2013-9, 307517/2014-4.} and Ariosvaldo M. Jatob\'a}
\date{}
\maketitle
\vspace*{-1.0em}
\begin{abstract}
\noindent In this work we shall prove new results on the theory of  convolution operators on spaces of entire functions. The focus is on
hypercyclicity results for convolution operators on spaces of entire functions of a given type and order; and existence and approximation results for convolution equations on spaces of entire functions of a given type and order.
In both cases we give a general method to prove new results that recover, as particular cases, several results of the literature. Applications of these more general results are given, including new hypercyclicity results for convolution operators on spaces on entire functions on $\mathbb{C}^n.$
\end{abstract}

\bigskip


\section{Introduction}
\qquad For a topological vector space $X$, a continuous linear operator $T \colon X \longrightarrow X$ is {\it hypercyclic} if the \emph{orbit of} $x$, given by
$\{x, T(x), T^2(x), \ldots\}$ is dense in $X$ for some $x \in X$.
In this case, $x$ is said to be a \emph{hypercyclic
vector for $T$}. Hypercyclic
translation and differentiation operators on spaces of entire functions of one
complex variable were first investigated by Birkhoff \cite{birkhoff} and
MacLane \cite{maclane}, respectively. Godefroy and Shapiro \cite{godefroy} pushed these
results quite further by proving that every nontrivial convolution operator on spaces of
entire functions of several complex variables is hypercyclic. By a \emph{nontrivial convolution operator} we mean a convolution operator which is not a scalar multiple of the identity. For the theory of hypercyclic operators and
its ramifications we refer to \cite{bay, goswinBAMS, grosse_peris}. We remark that several results on the hypercyclicity of operators on spaces of entire functions in infinitely many complex variables appeared later (see,
e.g., \cite{aron, bay2, favaro3, bes2012, CDSjmaa, chan, FM, GS,
goswinBAMS, MPS, peterssonjmaa}).
In 2007, Carando, Dimant and Muro \cite{CDSjmaa} proved some general results
that encompass as particular cases several of the above mentioned results. In
\cite{favaro3}, using the theory of holomorphy types, Bertoloto, Botelho,
F\'{a}varo and Jatob\'{a} generalized strictly the results of \cite{CDSjmaa} to a more
general setting. For instance, \cite[Theorem 2.7]{favaro3} recovers, as a very particular case,  the famous result of Godefroy
and Shapiro \cite{godefroy} on the hypercyclicity of convolution operators on
$\mathcal{H}(\mathbb{C}^{n})$.
%
%

The techniques of \cite{favaro3} are a refinement of a general method introduced in \cite{favaro1} to prove existence and approximation results for convolution
equations defined on the space $\mathcal{H}_{\Theta b}(E)$ of all entire functions of $\Theta$-bounded type defined on a complex Banach space $E$. 

The investigation of existence and approximation
results for convolution equations was initiated by
Malgrange \cite{malgrange} and developed by several authors (see, for instance
\cite{Col-Mat, cgp, cp, DW, DW2, favaro, FaBelg, favaro1, favaro2, gupta, G,
martineau, Matos-F, Matos-Z, Matos-Z2, Matos-livro, MN, Nach-B}). 

In this work we give contributions in two directions. In the first, we explore hypercyclicity results for convolution operators on the space $Exp_{\Theta,0,A}^{k}\left(  E\right) $, introduced in \cite{favaro2}, of $\Theta$ entire functions of a given type $A$ and order $k$ on a complex Banach space $E$, where $\Theta$ is a given holomorphy type. These results generalize the hypercyclicity results obtained in \cite{favaro3, birkhoff, CDSjmaa, godefroy, maclane}.

In the second direction, we obtain a general method to prove existence and approximation results for convolution equations on $Exp_{\Theta,0,A}^{k}\left(  E\right) $. These results generalize results of the same type obtained in \cite{favaro3, FaBelg, gupta, G, malgrange, martineau, matos3}. 

Both in the hypercyclicity results and in the existence and approximation results, the duality result via Fourier-Borel transform proved in \cite{favaro2} plays a central role. In the fashion
of Dineen \cite{dineen}, we identify the properties a holomorphy type must enjoy for the results to hold true (more precisely,  the $\pi_1$-$\pi_2$-holomorphy types introduced in \cite{favaro1} and refined in \cite{favaro3}, and the notion of $\pi_{2,k}$-holomorphy type introduced in Definition \ref{pi2k}). Moreover, our proofs of the hypercyclicity results rest on a classical hypercyclicity criterion, first obtained by Kitai \cite{kitai} and later on rediscovered by Gethner and Shapiro \cite{GS}. Nowadays it is known, by a result of Costakis and Sambarino \cite{costakis}, that the classical hypercyclicity criterion of Kitai ensures that the operator is mixing, a property stronger than hypercyclicity.  We recall that if $X$ is a topological vector space, then a continuous linear operator $T\colon X\rightarrow X$ is said to be \emph{mixing} if for any two non-empty open sets $U,V\subset X,$ there is $n_0\in \mathbb{N}$ such that $T^n(U)\cap V\neq\emptyset$ for all $n\geq n_0.$ 

This paper is organized as follows. In Section 2 we collect some general results which are often used in subsequent sections. In Section 3 we prove some preparatory results about convolution operators. In Section 4 we prove hypercyclicity results for convolution operators. Section 5 is devoted to the study of existence and approximation results for convolution equations. Finally, in Section 6 we provide new examples and show that well known examples are recovered by our results.

Throughout this paper $\mathbb{N}$ denotes the set of positive integers and
$\mathbb{N}_{0}$ denotes the set $\mathbb{N}\cup\{0\}$. By $\Delta$ we mean the open unit disk in the complex field $\mathbb{C}$. As usual, for
$k\in\left(  1,+\infty\right)  ,$ we denote by $k^{\prime}$ its conjugate,
that is, $\frac{1}{k}+\frac{1}{k^{\prime}}=1.$ For $k=1,$ we set $k^{\prime
}=+\infty.$ $E$ and $F$ are always complex Banach spaces and $E^{\prime}$ denotes the topological dual of $E$. The Banach space of all continuous $m$-homogeneous
polynomials from $E$ into $F$ endowed with its usual sup norm is denoted by $\mathcal{P}(^{m}E;F)$. The subspace of $\mathcal{P}(^{m}E;F)$ of all polynomials of finite type is represented by $\mathcal{P}_f(^{m}E;F)$. $\mathcal{H}(E;F)$ denotes the vector space of all holomorphic mappings from $E$ into $F$.
In all these cases, when $F = \mathbb{C}$ we write $\mathcal{P}(^{m}E)$, $\mathcal{P}_f(^{m}E)$  and $\mathcal{H}(E)$
instead of $\mathcal{P}(^{m}E;\mathbb{C})$, $\mathcal{P}_f(^{m}E;\mathbb{C})$  and $\mathcal{H}(E;\mathbb{C})$, respectively. For the general theory of homogeneous polynomials and holomorphic functions or any unexplained notation we refer to Dineen \cite{dineenlivro}, Mujica \cite{mujica} and Nachbin \cite{nachbin2}.

\section{Preliminaires}

\qquad We start recalling several concepts and results involving holomorphy on infinite dimensional spaces. 

\begin{definition}\label{def_holomorfia}
\rm Let $U$ be an open subset of $E$. A mapping $f\colon U\longrightarrow
F$ is said to be \textit{holomorphic on $U$} if for every $a\in U$ there
exists a sequence $(P_{m})_{m=0}^{\infty}$, where each $P_{m}\in
\mathcal{P}(^{m}E;F)$ ($\mathcal{P}(^{0}E;F)=F$), such that $f(x)=\sum\limits
_{m=0}^{\infty}P_{m}(x-a)$ uniformly on some open ball with center $a$. The
$m$-homogeneous polynomial $m!P_{m}$ is called the \textit{$m$-th derivative
of $f$ at $a$} and is denoted by $\hat{d}^{m}f(a)$. In particular, if
$P\in\mathcal{P}(^{m}E;F)$, $a\in E$ and $k\in\{0,1,\ldots,m\}$, then
\[
\hat{d}^{k}P(a)(x)=\frac{m!}{(m-k)!}\check{P}(\underbrace{x,\ldots
,x}_{k\,{times}},a,\ldots,a)
\]
for every $x\in E$, where $\check{P}$ is the unique symmetric $m$-linear
mapping associated to $P$.
\end{definition}

\begin{definition}
\rm\label{holomorphy type} (Nachbin \cite{nachbin2}) A \emph{holomorphy type $\Theta$ from $E$ to $F$} is a sequence of Banach
spaces $(P_{\Theta}(^{j} E;F))_{j=0}^{\infty}$, the norm on each of them being
denoted by $\|\cdot\|_{\Theta}$, such that the following conditions hold true:

\begin{enumerate}
\item[$(1)$] \textrm{Each $P_{\Theta}(^{j} E;F)$ is a vector subspace of $P(^{j}
E; F)$ and $P_{\Theta}(^{0} E;F)$ coincides with $F$ as
a normed vector space;}

\item[$(2)$] \textrm{There is a real number $\sigma\geq1$ for which the
following is true: given any $k\in\mathbb{N}_{0}$, $j\in\mathbb{N}_{0}$,
$k\leq j $, $a\in E$, and $P\in\mathcal{P}_{\Theta}(^{j}E;F)$, we have
$\hat{d}^{k}P(a)\in\mathcal{P}_{\Theta}(^{k}E;F)$ and
\[
\left\|  \frac{1}{k!}\hat{d}^{k}P(a)\right\|  _{\Theta}\leq\sigma
^{j}\|P\|_{\Theta}\|a\|^{j-k}.
\]
}
\end{enumerate}
A holomorphy type from $E$ to $F$ shall be denoted by either $\Theta$ or $( \mathcal{P}_{\Theta}(^{j}E;F))_{j=0}^\infty$. When $F=\mathbb{C}$ we write $P_\Theta(^{j}E)$ instead of $P_{\Theta}(^{j} E;\mathbb{C})$, for every $j\in\mathbb{N}_0.$
\end{definition}
\noindent It is obvious that each inclusion $\mathcal{P}_{\Theta}(^{j}E;F)\subset
\mathcal{P}(^{j}E;F)$ is continuous and $\|P\|\leq\sigma^{j}\|P\|_{\Theta}$.

\begin{definition}\rm(\cite[Definition 2.2]{favaro2})
\textrm{\label{bk} Let $(\mathcal{P}_{\Theta}(^{j}E))_{j=0}^{\infty}$ be a
holomorphy type from $E$ to $\mathbb{C}$. For $\ \rho>0$ and $k\geq1,$ we
denote by $\mathcal{B}_{\Theta,\rho}^{k}\left(  E\right)  $ the complex Banach  space of all $f\in\mathcal{H}\left(  E\right)  $ such that $\widehat{d}%
^{j}f\left(  0\right)  \in\mathcal{P}_{\Theta}\left(  ^{j}E\right)  ,$ for all
$j\in\mathbb{N}_{0}$\ and%
\[
\left\Vert f\right\Vert _{\Theta,k,\rho}={\displaystyle\sum\limits_{j=0}%
^{\infty}} \rho^{-j}\left(  \frac{j}{ke}\right)  ^{\frac{j}{k}}\left\Vert
\frac{1}{j!}\widehat{d}^{j}f\left(  0\right)  \right\Vert _{\Theta}<+\infty,
\]
with the norm given by $\left\Vert \cdot\right\Vert _{\Theta,k,\rho}.$ }
\end{definition}


Now we recall the definition of the spaces of entire functions of a given type $A$ and finite order $k$.

\begin{definition}\rm(\cite[Definition 2.4]{favaro2})
\label{expk}Let $\left(  \mathcal{P}_{\Theta}(^{j}E)\right)
_{j=0}^{\infty}$ be a holomorphy type from $E$ to $\mathbb{C}$,
$A\in\left[  0,+\infty\right)  $ and $k\geq1.$ We denote by $Exp_{\Theta,0
,A}^{k}\left(  E\right)  $ the complex vector space ${\displaystyle\bigcap
\limits_{\rho>A}} \mathcal{B}_{\Theta,\rho}^{k}\left(  E\right)  $ with the locally convex projective limit topology. In case $A=0$ we denote $Exp_{\Theta,0}^{k}\left(
E\right)  :=Exp_{\Theta ,0,0}^{k}\left(  E\right)  =
{\displaystyle\bigcap\limits_{\rho>0}}
\mathcal{B}_{\Theta,\rho}^{k}\left(  E\right)  $. By \cite[Proposition 2.7]{favaro2} $Exp_{\Theta,0
,A}^{k}\left(  E\right)  $  is a Fr\'echet space.
\end{definition}

%


\begin{proposition}(\cite[Proposition 2.5]{favaro2})
\label{caracterizacao sm(r,q)}Let $\left(  \mathcal{P}_{\Theta}(^{j}E)\right)
_{j=0}^{\infty}$ be a holomorphy type from $E$ to $\mathbb{C}$ and
$k\in[1,+\infty)$. If $f\in$ $\mathcal{H}\left(  E\right)  $ is such that
$\widehat{d}^{j}f\left(  0\right)  \in\mathcal{P}_{\Theta}\left(
^{j}E\right)  ,$ $\forall j\in\mathbb{N}_{0}$, then for each $A\in\left[  0,+\infty\right), f\in Exp_{\Theta,0,A}^{k}\left(  E\right)  $ if, and only if, 
$$\limsup
\limits_{j\rightarrow\infty}\left(  \frac{j}{ke}\right)  ^{\frac{1}{k}%
}\left\Vert \frac{1}{j!}\widehat{d}^{j}f\left(  0\right)  \right\Vert
_{\Theta}^{\frac{1}{j}}\leq A.$$
\end{proposition}

Now we recall the definition of the spaces of holomorphic functions of a given type $A$ and infinite order.

\begin{definition}\rm(\cite[Definition 2.8]{favaro2})
\rm{\label{expInfinito}Let $\left(  \mathcal{P}_{\Theta}(^{j}E)\right)
_{j=0}^{\infty}$ be a holomorphy type from $E$ to $\mathbb{C}$. If
$A\in\left[  0,+\infty\right)  ,$ we denote by $\mathcal{H}_{\Theta b}\left(
B_{\frac{1} {A}}\left(  0\right)  \right)  $ the Fr\'echet space of all
$f\in\mathcal{H}\left(  B_{\frac{1}{A}}\left(  0\right)  \right)  $ such that
$\widehat{d}^{j}f\left(  0\right)  \in\mathcal{P}_{\Theta}\left(
^{j}E\right)  ,$ for all $j\in\mathbb{N}_{0}$ and}
\[
\limsup\limits_{j\rightarrow\infty}\left\Vert \frac{1}{j!}\widehat{d}%
^{j}f\left(  0\right)  \right\Vert _{\Theta}^{\frac{1}{j}}\leq A,
\]
\textrm{endowed with the locally convex topology generated by the family of
seminorms $\left(  p_{\Theta,\rho}^{\infty}\right)  _{\rho>A},$ where
\[
p_{\Theta,\rho}^{\infty}\left(  f\right)  = {\displaystyle\sum\limits_{j=0}%
^{\infty}} \rho^{-j}\left\Vert \frac{1}{j!}\widehat{d}^{j}f\left(  0\right)
\right\Vert _{\Theta}.
\]
}
\end{definition}

We also denote $\mathcal{H}_{\Theta b}\left(  B_{\frac{1}{A}}\left(  0\right)
\right)  $ by $Exp_{\Theta,0,A}^{\infty}\left(  E\right)  $ and we also write
$Exp_{\Theta,0}^{\infty}\left(  E\right)  =Exp_{\Theta,0,0}^{\infty}\left(
E\right)  $. 

\begin{remark}
\rm Note that the space $Exp_{\Theta,0}^{\infty}\left(  E\right)$ coincides with the space $\mathcal{H}_{\Theta b}(E)$ introduced by Nachbin \cite{nachbin2}. 
\end{remark}
\medskip

We need to recall the algebraic isomorphism of the Fourier-Borel transform proved in \cite[Theorems 4.6 and 4.9]{favaro2}. This isomorphism plays a key role in the proofs of the hypercyclicity results for convolution operators and existence and approximation results for convolution equations. To introduce the Fourier-Borel transform it is necessary to recall the definition of the Borel transform and to use that $\Theta$ is a  $\pi_{1}$-holomorphy type. The concept of  $\pi_{1}$-holomorphy type was originally introduced in \cite[Definitions 2.3]{favaro1} and nowadays we use a slight variation of this concept which can be found in \cite[Definition 2.5]{favaro3}. This concept and related notions are very useful to prove results of this type, see e.g. \cite{favaro3, CDSjmaa, favaro1, favaro2,FM, gupta, G, MPS}.

\begin{definition}\label{pi-tipo de holomorfia}
\rm Let $(\mathcal{P}_{\Theta}(^{j}E;F))_{j=0}^{\infty}$ be a
holomorphy type from $E$ to $F$. We say that this holomorphy type is a \emph{$\pi_1$-holomorphy type} if the following conditions hold:
\begin{enumerate}
\item[($1$)]  $P_f(^{j}E;F) \subset \mathcal{P}_{\Theta}(^{j}E;F)$ and there exists $K>0$ such that
$\|\phi^j\cdot b\|_{\Theta}\leq K^ j\|\phi\|^j\|b\|,$ for all $\phi\in E^{\prime}$, $b\in F$ and $j\in \mathbb{N}_0$;
\item[($2$)] For $j\in \mathbb{N}_0$, $P_f(^{j}E;F)$ is dense in $\left(\mathcal{P}_{\Theta}(^{j}E;F), \|\cdot \|_{\Theta}\right)$.
\end{enumerate}
\end{definition}

Now we recall the isomorphism given by Borel transform (see \cite[Definition 4.1]{favaro3} or \cite[p. 915]{favaro1}).

\begin{definition}\rm Let $\Theta$ be a $\pi_{1}$-holomorphy type from $E$ to
$F$. It is clear that the {\it Borel transform}
\[
\mathcal{B}_{\Theta}\colon\left[  \mathcal{P}_{\Theta}(^{m}E;F)\right]
^{\prime}\longrightarrow\mathcal{P}(^{m}E^{\prime};F^{\prime})~, ~ \mathcal{B}_{\Theta} T(\phi)(y)=T(\phi^{m}y),\] for $T\in\left[
\mathcal{P}_{\Theta}(^{m}E;F)\right]  ^{\prime}$, $\phi\in E^{\prime}$ and
$y\in F$, is well defined and linear. Moreover, ${\cal B}_\Theta$ is continuous and injective by conditions (a1) and (a2) of Definition \ref{pi-tipo de holomorfia}. So, denoting the range of $\mathcal{B}_{\Theta}$ in $\mathcal{P}%
(^{m}E^{\prime};F^{\prime})$ by $\mathcal{P}_{\Theta^{\prime}}(^{m}E^{\prime
};F^{\prime})$, the correspondence
$$\mathcal{B}_{\Theta}
T \in \mathcal{P}%
_{\Theta^{\prime}}(^{m}E^{\prime};F^{\prime})\mapsto \|\mathcal{B}_{\Theta}
T\|_{\Theta^{\prime}}:=\|T\|,  $$
defines a norm on $\mathcal{P}%
_{\Theta^{\prime}}(^{m}E^{\prime};F^{\prime})$. 

In this fashion the spaces $\left(  \left[  \mathcal{P}_{\Theta}%
(^{m}E;F)\right]  ^{\prime}\;,\|\cdot\|\right)  $ and $(\mathcal{P}_{\Theta^{\prime}}(^{m}E^{\prime};F^{\prime}),\;\|\cdot
\|_{\Theta^{\prime}})$ are isometrically isomorphic.
\end{definition}

\begin{definition}\rm(\cite[Definition 4.3]{favaro2})
\label{bklinha} Let $(\mathcal{P}_{\Theta}(^{j}E))_{j=0}^{\infty}$ be
a $\pi_{1}$-holomorphy type from $E$ to $\mathbb{C}$. If $\ \rho>0$ and
$k\geq1,$ we denote by $\mathcal{B}_{\Theta^{\prime},\rho}^{k}\left(
E^{\prime}\right)  $ the complex Banach space of all $f\in\mathcal{H}\left(
E^{\prime}\right)  $ such that $\widehat{d} ^{j}f\left(  0\right)
\in\mathcal{P}_{\Theta^{\prime}}\left(  ^{j}E^{\prime}\right)  ,$ for all
$j\in\mathbb{N}_{0}$\ and
\[
\left\Vert f\right\Vert _{\Theta^{\prime},k,\rho}={\displaystyle\sum
\limits_{j=0} ^{\infty}} \rho^{-j}\left(  \frac{j}{ke}\right)  ^{\frac{j}{k}%
}\left\Vert \frac{1}{j!}\widehat{d}^{j}f\left(  0\right)  \right\Vert
_{\Theta^{\prime}}<+\infty.
\]
\end{definition}

As done just below of \cite[Definition 4.3]{favaro2} we may consider the space $Exp_{\Theta^\prime
,A}^{k}\left(  E\right)  $, for every $k\in[1,+\infty]$ and $A\in\left(  0,+\infty\right]$, but in this paper we are particularly interest in case $A=+\infty$. We recall the definition now.

\begin{definition}
\rm Let $\left(  \mathcal{P}_{\Theta}(^{j}E)\right)
_{j=0}^{\infty}$ be a $\pi_1$-holomorphy type from $E$ to $\mathbb{C}$ and $A\in\left(  0,+\infty\right]$. 

\item[(a)] For $k\geq1,$ we denote by $Exp_{\Theta^\prime
,A}^{k}\left(  E\right)  $ the complex vector space ${\displaystyle\bigcup
\limits_{\rho<A}} \mathcal{B}_{\Theta^\prime,\rho}^{k}\left(  E\right)  $ with the
locally convex inductive limit topology. 

\item[(b)] For $k=+\infty,$ $Exp_{\Theta^\prime
,\infty}^{\infty}\left(  E^\prime\right)  $ is a space of germs of holomorphic functions described in the following way:  For $\rho>0,$
we define the complex vector space $\mathcal{H}_{\Theta^\prime}^{\infty}\left(
B_{\frac{1}{\rho}}\left(  0\right)  \right)  $ of all $f\in\mathcal{H}\left(
B_{\frac{1}{\rho}}\left(  0\right)  \right)  ,$  where $B_{\rho}\left(  0\right)  $ denotes the open ball centered in $0$ and radius $\rho$ in $E^{\prime
},$ such that $\widehat{d}
^{j}f\left(  0\right)  \in\mathcal{P}_{\Theta^\prime}\left(  ^{j}E\right)  ,$ for all
$j\in\mathbb{N}_{0}$\ and%
\[
p_{\Theta^\prime,\rho}^{\infty}={\displaystyle\sum\limits_{j=0}^{\infty}} \rho^{-j}\left\Vert \frac{1}%
{j!}\widehat{d}^{j}f\left(  0\right)  \right\Vert _{\Theta^\prime}<+\infty,
\]
which is a Banach space with the norm $p_{\Theta^\prime,\rho}^{\infty}.$

Let $S= {\displaystyle\bigcup\limits_{0<\rho< A}} \mathcal{H}_{\Theta^\prime }^\infty\left(
B_{\frac{1}{\rho}}\left(  0\right)  \right)  $ and define the following
equivalence relation:%
\[
f\sim g\Longleftrightarrow\text{ there is }\rho>0  \text{
such that }f|_{B_{\frac{1}{\rho}}\left(  0\right)  }=g|_{B_{\frac{1}{\rho}%
}\left(  0\right)  .}
\]
We denote by $S\left/\sim\right.  $ the set of all equivalence classes of elements of $S$ and by
$\left[  f\right]  $ the equivalence class which has $f$ as one representative.
If we define the operations
\[
\left[  f\right]  +\left[  g\right]  =\left[  f|_{B_{\frac{1}{\rho}}\left(
0\right)  }+g|_{B_{\frac{1}{\rho}}\left(  0\right)  }\right]  ,
\]
where $\rho\in (0,A)  $ is such that $f|_{B_{\frac{1}{\rho}%
}\left(  0\right)  },g|_{B_{\frac{1}{\rho}}\left(  0\right)  }\in
\mathcal{H}_{\Theta^\prime  }^\infty\left(  B_{\frac{1}{\rho}}\left(  0\right)  \right)  ,$
and
\[
\lambda\left[  f\right]  =\left[  \lambda f\right]  ,\text{\qquad}\lambda
\in\mathbb{C},
\]
then $S\left/  \sim\right.  $ becomes a vector space. For each $\rho\in (0,A)  $, let $i_{\rho}\colon\mathcal{H}_{\Theta^\prime
}^\infty\left(  B_{\frac{1}{\rho}}\left(  0\right)  \right)  \longrightarrow
S\left/  \sim\right.  $ be given by $i_{\rho}\left(  f\right)  =\left[
f\right]  .$ So we define $Exp_{\Theta^\prime
,A}^{\infty}\left(  E\right) $  being the space $S\left/  \sim\right.  $ with the locally convex inductive
limit topology generated by the family $\left(  i_{\rho}\right)  _{\rho\in (0,A)}$. 

In both cases, $Exp_{\Theta^\prime
,A}^{k}\left(  E\right)  $ becomes a $DF$-space.
\end{definition}

Now we are able to recall the algebraic isomorphism given by the Fourier-Borel transform $\mathcal{F}$ in \cite[Theorem 4.6]{favaro2}. We define
$\lambda\left(  k\right)  =\frac{k}{\left(  k-1\right)  ^{\frac{k-1}{k}}},$
for$\ k\in\left(  1,+\infty\right)  .$ Since $\lim\limits_{k\rightarrow\infty}\lambda\left(
k\right) =1 ,$ we set $\lambda\left(  \infty\right)
=1.$ When $A=0$ we write $A^{-1}=+\infty.$

\begin{theorem}\label{fourier_borel}
Let $k\in\left(  1,+\infty\right]  $,  $A\in\left[  0,+\infty\right)  $ and $(\mathcal{P}_{\Theta}(^{j}E))_{j=0}^{\infty}$
be a $\pi_{1}$-holomorphy type from $E$ to $\mathbb{C}$.
%
Then the Fourier-Borel transform%
\[
\mathcal{F}\colon\left[  Exp_{\Theta  ,0, A}%
^{k}\left(  E\right)  \right]  ^{\prime}\longrightarrow Exp_{\Theta^{\prime}  , (\lambda(k)A)^{-1}}^{k^{\prime}}\left(  E^{\prime
}\right)  ,
\]
given by $\mathcal{F}T\left(  \varphi\right)  =T\left(  e^{\varphi}\right)  ,$ for all
$T\in\left[  Exp_{\Theta  ,0,A}%
^{k}\left(  E\right)  \right]  ^{\prime}$\ and $\varphi\in E^{\prime},$
establishes an algebraic isomorphism.

Note that, when $A=0$, we have
$\mathcal{F}$ from $\left[  Exp_{\Theta  ,0}%
^{k}\left(  E\right)  \right]  ^{\prime}$ to $Exp_{\Theta^{\prime}  , \infty}^{k^{\prime}}\left(  E^{\prime
}\right) .$

%
\end{theorem}

%

\section{Convolution operators}
We start with a preliminary result we need to
introduce convolution operators on $Exp_{\Theta  ,0}^{k}\left( E\right)$.

\begin{proposition}\label{derivadas}
\label{dnfa}Let $\left(  \mathcal{P}_{\Theta}(^{j}E)\right)
_{j=0}^{\infty}$ be a holomorphy type from $E$ to $\mathbb{C}$,
 $a\in E,$ $k\in\left[  1,+\infty\right]  ,$ $A\in\left[
0,+\infty\right)  $ and $f\in Exp_{\Theta  ,0,A}^{k}\left( E\right) .$ Then
$\widehat{d}^{n}f\left(  \cdot\right)  a\in Exp_{\Theta ,0,\sigma A}^{k}\left(  E\right)  ,$ for any constant $\sigma$ satisfying condition (3) of Definition \ref{holomorphy type}. Besides
{\small
\[
\widehat{d}^{n}f\left(  \cdot\right)  a=%
{\displaystyle\sum\limits_{j=0}^{\infty}}
\frac{1}{ j!}\overset
{\begin{picture}(60,5)\put(0,0){\line(6,1){30}}\put(30,5){\line(6,-1){30}}\end{picture}}%
{d^{j+n}f\left(  0\right)  \cdot^{j}}\left(  a\right)  ,
\]}
in the topology of $Exp_{\Theta  ,0,\sigma A}^{k}\left(  E\right)  .$
\end{proposition}

\begin{proof}
It is known (see Nachbin \cite[p. 29]{nachbin2}) that, for a fixed $j\in\mathbb{N}_0,$
{\small
\begin{equation}
\widehat{d}^{j}f\left(  x\right)  a=%
{\displaystyle\sum\limits_{n=0}^{\infty}} \frac{1}{ n!}\overset
{\begin{picture}(60,5)\put(0,0){\line(6,1){30}}\put(30,5){\line(6,-1){30}}\end{picture}}%
{d^{j+n}f\left(  0\right)  x^{n}}\left(  a\right)  =%
{\displaystyle\sum\limits_{n=0}^{\infty}} \frac{1}{ n!}\overset
{\begin{picture}(60,5)\put(0,0){\line(6,1){30}}\put(30,5){\line(6,-1){30}}\end{picture}}%
{d^{j+n}f\left(  0\right)  a^{j}}\left(  x\right)  , \label{serieConvPontual}%
\end{equation}
}%
for all $x\in E$. Since $\hat{d}^{m}f(0)\in\mathcal{P}_{\Theta}(^{m}E),$ for all
$m\in \mathbb{N}_{0},$ then $\Hhat{d^{j+n}f(0)a^j}\in
\mathcal{P}_{\Theta}(^{n}E)$ and%
{\small
\begin{equation}
\label{Teo1_desig1}\left\Vert\Hhat{d^{j+n}f(0)a^j}\right\Vert_{\Theta}\leq\frac{n!j!\sigma^{j+n}}%
{(j+n)!}\left\Vert\hat{d}^{j+n}f(0)\right\Vert_{\Theta}\Vert
a\Vert^{j},
\end{equation}
}

\noindent for all $n\in\mathbb{N}_0.$ In fact, let $P$ being the $(n+j)$-homogeneous polynomial $\hat{d}^{j+n}f(0)$. Thus $\check{P}=d^{j+n}f(0)$ and it follows from Definition \ref{def_holomorfia} that
\small{\[
\Hhat{d^{j+n}f(0)a^j}=\widehat{\check{P}a^j}=\frac{j!}{(n+j)!}\hat{d}^n P(a).
\]}

Using condition $(3)$ of Definition $\ref{holomorphy type}$ we have

\small{\[
\left\Vert\Hhat{d^{j+n}f(0)a^j}\right\Vert_{\Theta}=\frac{j!}{(n+j)!}\left\Vert \hat{d}^n P(a)\right\Vert_{\Theta}\leq\frac{n!j!\sigma^{j+n}}%
{(j+n)!}\left\Vert P\right\Vert_{\Theta}\Vert
a\Vert^{j}=\frac{n!j!\sigma^{j+n}}%
{(j+n)!}\left\Vert\hat{d}^{j+n}f(0)\right\Vert_{\Theta}\Vert
a\Vert^{j}.
\]}

For $k\in\left[  1,+\infty\right)  ,$ let%
{\small
\[
N=\limsup\limits_{n\rightarrow\infty}\left(  \frac{n+j}{ke}\right)
^{\frac {1}{k}}\left\Vert \frac{\widehat{d}^{n+j}f\left(  0\right)
}{\left( n+j\right)  !}\right\Vert _{\Theta }^{\frac{1}{n+j}}.
\]}%
By Proposition \ref{caracterizacao sm(r,q)} we have $N\leq A<+\infty$. Then for every $\varepsilon>0,$ there is $C\left(
\varepsilon\right)  >0$
such that%
{\small
\begin{equation}
\label{Teo1_desig2}\left(  \frac{n+j}{ke}\right)  ^{\frac{n+j}{k}}\left\Vert
\frac{\widehat {d}^{n+j}f\left(  0\right)  }{\left(  n+j\right)
!}\right\Vert _{\Theta
}\leq C\left(  \varepsilon\right) \left(  N+\varepsilon\right)
^{n+j},
\end{equation}}%
for all  $n\in\mathbb{N}_0$. Hence%
{\small
\begin{gather}
\left(  \frac{n}{ke}\right)  ^{\frac{n}{k}}\frac{1}{n!}\left\Vert
\overset
{\begin{picture}(60,5)\put(0,0){\line(6,1){30}}\put(30,5){\line(6,-1){30}}\end{picture}}%
{d^{j+n}f\left(  0\right)  a^{j}}\right\Vert _{\Theta}\overset
{(\ref{Teo1_desig1})}{\leq}\left(  \frac{n}{ke}\right)  ^{\frac{n}{k}}\frac{j!\sigma^{j+n}}{(j+n)!}\left\Vert \widehat{d}%
^{n+j}f\left(  0\right)  \right\Vert _{\Theta}\left\Vert a\right\Vert
^{j}\nonumber\\
=j!\sigma^{j+n}\left(  \frac{n}{ke}\right)  ^{\frac{n}{k}}\left\Vert
\frac{\widehat{d}^{n+j}f\left(  0\right)  }{\left(  n+j\right)  !}\right\Vert
_{\Theta}\left\Vert a\right\Vert ^{j}\overset{(\ref{Teo1_desig2})}{\leq
}j!\sigma^{j+n}\left(  \frac{n}{ke}\right)  ^{\frac{n}{k}}\left(  \frac
{ke}{n+j}\right)  ^{\frac{n+j}{k}}C\left(  \varepsilon\right)  \left(
N+\varepsilon\right)  ^{n+j}\left\Vert a\right\Vert ^{j}\nonumber\\
=j!\sigma^{j+n}\left(  \frac{n}{n+j}\right)  ^{\frac{n}{k}}\left(  \frac
{ke}{n+j}\right)  ^{\frac{j}{k}}C\left(  \varepsilon\right)  \left(
N+\varepsilon\right)  ^{n+j}\left\Vert a\right\Vert ^{j}.\label{des1}%
\end{gather}
}%
Since%
{\small
\[
\lim_{n\rightarrow\infty}\left(  j!\right)  ^{\frac{1}{n}}\sigma^{\frac{j}%
{n}+1}\left(  \frac{n}{n+j}\right)  ^{\frac{1}{k}}\left(  \frac{ke}%
{n+j}\right)  ^{\frac{j}{kn}}=\sigma,
\]
}%
there is $D\left(  \varepsilon\right)  >0$ such that
{\small
\begin{equation}
j!\sigma^{j+n}\left(  \frac{n}{n+j}\right)  ^{\frac{n}{k}}\left(  \frac
{ke}{n+j}\right)  ^{\frac{j}{k}}\leq D\left(
\varepsilon\right)  \left(  \sigma+\varepsilon\right)  ^{n}, \label{des2}%
\end{equation}}%
\noindent for all  $n\in\mathbb{N}_0$. From (\ref{des1}) and (\ref{des2}) we obtain%
{\small
\[
\left(  \frac{n}{ke}\right)  ^{\frac{n}{k}}\frac{1}{n!}\left\Vert
\overset
{\begin{picture}(60,5)\put(0,0){\line(6,1){30}}\put(30,5){\line(6,-1){30}}\end{picture}}%
{d^{j+n}f\left(  0\right)  a^{j}}\right\Vert _{\Theta}\leq C\left(  \varepsilon\right)
D\left( \varepsilon\right)  \left\Vert a\right\Vert ^{j}\left(
N+\varepsilon\right) ^{j}\left[  \left(  \sigma +\varepsilon\right)
\left(  N+\varepsilon\right) \right]  ^{n},
\]}%
for all $n\in\mathbb{N}_0$ and $\varepsilon>0.$ Therefore%
{\small
\[
\limsup\limits_{n\rightarrow\infty}\left(  \frac{n}{ke}\right)  ^{\frac{1}{k}%
}\left\Vert \frac{\overset
{\begin{picture}(60,5)\put(0,0){\line(6,1){30}}\put(30,5){\line(6,-1){30}}\end{picture}}%
{d^{j+n}f\left(  0\right)  a^{j}}}{n!}\right\Vert
_{\Theta
}^{\frac{1}{n}}\leq\left(  \sigma+\varepsilon \right)  \left(
N+\varepsilon\right)  ,
\]}%
for all $\varepsilon>0$, which implies%
{\small
\[
\limsup\limits_{n\rightarrow\infty}\left(  \frac{n}{ke}\right)  ^{\frac{1}{k}%
}\left\Vert \frac{\overset
{\begin{picture}(60,5)\put(0,0){\line(6,1){30}}\put(30,5){\line(6,-1){30}}\end{picture}}%
{d^{j+n}f\left(  0\right)  a^{j}}}{n!}\right\Vert
_{\Theta
}^{\frac{1}{n}}\leq\sigma N.
\]}%

\noindent Since $N\leq A$, then $\sigma N\leq\sigma A$ and so $\widehat{d}^{n}f\left(  \cdot\right)  a\in
Exp_{\Theta
,0,\sigma A}^{k}\left( E\right)  $.
\newline Now we consider $k=+\infty.$ If $A=0$, then
$\mathcal{H}_{\Theta b
}\left(  E\right) =Exp_{\Theta ,0}^{\infty}\left(  E\right)  $ and this case was proved in \cite[Proposition 3.1 (i)]{favaro1}. If $A\neq 0$, then for $f\in Exp_{\Theta ,0,A}^{\infty}\left(  E\right)
=\mathcal{H}_{\Theta b
}\left(  B_{\frac{1}{A}}\left(
0\right)  \right)  $ we have%
{\small
\[
\limsup\limits_{n\rightarrow\infty}\left\Vert
\frac{\widehat{d}^{n+j}f\left( 0\right)  }{\left(  n+j\right)
!}\right\Vert _{\Theta
}^{\frac{1}{n+j}}\leq  A
\]}%
and as above we obtain
{\small
\[
\limsup\limits_{n\rightarrow\infty}\left\Vert \frac{\overset
{\begin{picture}(60,5)\put(0,0){\line(6,1){30}}\put(30,5){\line(6,-1){30}}\end{picture}}%
{d^{j+n}f\left(  0\right)  a^{j}}}{n!}\right\Vert
_{\Theta
}^{\frac{1}{n}}\leq \sigma A.
\]}%
Thus $\widehat{d}^{n}f\left(  \cdot\right) a\in\mathcal{H}_{\Theta b }\left(
B_{\frac{1}{\sigma A}}\left( 0\right)  \right) =Exp_{\Theta ,0,\sigma A}^{\infty}\left(  E\right)  .$

Now we only have to prove the
convergence of the series in the topology of $Exp_{\Theta
,0,\sigma A}^{k}\left( E\right)  $. If
$f\in\mathcal{B}_{\Theta
,\rho}^{k}\left(  E\right)  $ for some $\rho>0$ with $k\in\left[
1,+\infty
\right)  ,$ repeating the argument above with $\rho$ instead of $L$ we obtain constants $C_{1}\left(  \varepsilon\right)>0$ and $D_{1}\left(  \varepsilon\right)>0$ such that%
{\small
\begin{gather*}
\left\Vert \widehat{d}^{j}f\left(  \cdot\right)  a-%
{\displaystyle\sum\limits_{n=0}^{v}}
\left(  n!\right)  ^{-1}\overset
{\begin{picture}(60,5)\put(0,0){\line(6,1){30}}\put(30,5){\line(6,-1){30}}\end{picture}}%
{d^{j+n}f\left(  0\right)  \cdot^{n}}\left(  a\right)  \right\Vert
_{\Theta  ,k,\rho_{0}}
\leq%
{\displaystyle\sum\limits_{n=v+1}^{\infty}}
\rho_{0}^{-n}\left(  \frac{n}{ke}\right)  ^{\frac{1}{k}}\left\Vert
\left( n!\right)  ^{-1}\widehat{d}^{j+n}f\left(  0\right)
\right\Vert _{\Theta  }\left\Vert a\right\Vert ^{j}\sigma^{ n+j}\\
\leq C_{1}\left(  \varepsilon\right)  D_{1}\left(  \varepsilon\right)
\left\Vert
a\right\Vert ^{j}\left(  \rho+\varepsilon\right)  ^{j}%
{\displaystyle\sum\limits_{n=v+1}^{\infty}}
\left[  \rho_{0}^{-1}\left(  \rho+\varepsilon\right)  \left(
\sigma+\varepsilon \right)  \right]  ^{n},
\end{gather*}
}%
and this tends to zero when $v\rightarrow\infty,$ for
$\rho_{0}>\rho$ and $\varepsilon>0$ such that $\left(
\rho+\varepsilon\right) \left( \sigma+\varepsilon\right)  <\rho_{0}$. Hence $
{\displaystyle\sum\limits_{j=0}^{\infty}}
\frac{1}{ j!}\overset
{\begin{picture}(60,5)\put(0,0){\line(6,1){30}}\put(30,5){\line(6,-1){30}}\end{picture}}%
{d^{j+n}f\left(  0\right)  \cdot^{j}}\left(  a\right)
$
converges to $\widehat{d}^{n}f\left(  \cdot\right)  a$ in the topology of $Exp_{\Theta
,0,\sigma A}^{k}\left( E\right)  .$ The case $k=+\infty$ is
analogous.
\end{proof}

Now we restrict ourselves to the case $A=0.$ The case for an arbitrary $A$ shall be treated later (cf. Theorem \ref{main3}).
The following concept is well defined due to Proposition \ref{derivadas}.

\begin{definition}
\rm\textrm{\label{defOperConv} For $k\in\left[
1,+\infty\right] ,$ a
\emph{convolution operator on }$Exp_{\Theta ,0}^{k}\left( E\right) $ is a continuous
linear mapping {\small
\[ L\colon Exp_{\Theta  ,0}^{k}\left(  E\right)
\longrightarrow Exp_{\Theta ,0}^{k}\left(  E\right)
\]}%
such that $d\left( Lf\right)  \left( \cdot\right)
a=L\left( df\left(  \cdot\right)  a\right)  $ for all
$a\in E$ and $f\in Exp_{\Theta ,0}^{k}\left(  E\right)  .$ We denote the set of
all convolution
operators on $Exp_{\Theta  ,0}
^{k}\left(  E\right)  $ by  $\mathcal{A}_{\Theta
,0}^{k}.$}
\end{definition}

\medskip

Using induction it is easy to check that
convolution operators commute with all the directional derivatives
of all orders, that is, for all $a\in E, n\in\mathbb{N} $ and $L\in\mathcal{A}_{\Theta
,0}^{k}$,
$L\left(  \widehat{d}^{n}f\left( \cdot\right)  a\right)
=\widehat{d}^{n}\left(  Lf\right) \left(  \cdot\right)
a.$ In Theorem \ref{teorema_translacao} we shall prove that convolution operators could have been defined replacing the condition
$d\left(  Lf\right)  \left( \cdot\right)  a=L\left(
df\left(  \cdot\right)  a\right)  $ by $\tau_{-a}\left(
L\left( f\right)  \right) =L\left(
\tau_{-a}f\right)  $ for all $a\in E,$ where $\tau_{-a}f\left(
x\right)  =f\left(  x+a\right)  ,$ for all $x\in E.$ So, commutativity with the directional
derivatives or translations are equivalent concepts. First we need to prove that the translations are well-defined.

\begin{proposition}
\label{traslacao}

Let $\left(  \mathcal{P}_{\Theta}(^{j}E)\right)
_{j=0}^{\infty}$ be a holomorphy type from $E$ to $\mathbb{C}$ and $k\in\left[ 1,+\infty\right]  .$ If $f\in Exp_{\Theta
,0}^{k}\left( E\right)  $ and $a\in E,$ then $\tau_{-a}f\in
Exp_{\Theta ,0}^{k}\left( E\right)
$ and%
{\small
\[
\tau_{-a}f=%
{\displaystyle\sum\limits_{n=0}^{\infty}}
\frac{1}{n!}\widehat{d}^{n}f\left(  \cdot\right)  a,
\]}%
in the topology of $Exp_{\Theta  ,0}^{k}\left(  E\right)  .$
\end{proposition}

\begin{proof}
The case $k=+\infty$ was proved in
\cite[Proposition 3.1 (ii)]{favaro1}. For
$k\in\left[  1,+\infty\right)  $ and $f\in Exp_{\Theta,0}
^{k}\left(  E\right)  $ we have that
{\small
\begin{equation*}
\limsup\limits_{j\rightarrow\infty}\left(  \frac{j}{ke}\right)  ^{\frac{1}{k}%
}\left\Vert \frac{\widehat{d}^{j}f\left(  0\right)
}{j!}\right\Vert _{\Theta }^{\frac{1}{j}}= 0.
\label{limsup3}%
\end{equation*}
}%
Thus for all $\varepsilon>0$ there is $C\left( \varepsilon\right)
>0$ such
that%
{\small
\begin{equation}
\left(  \frac{j}{ke}\right)  ^{\frac{j}{k}}\left\Vert \frac{\widehat{d}%
^{j}f\left(  0\right)  }{j!}\right\Vert _{\Theta }\leq C\left(
\varepsilon\right)   \varepsilon^{j}, \label{epsilon+limsup}%
\end{equation}
}%
for all  $j\in\mathbb{N}$. Since, for each $n\in\mathbb{N}$, $\widehat{d}^{n}\left(
\tau_{-a}f\right)  \left( 0\right)
=\widehat{d}^{n}f\left(  a\right)  $ then we have%
{\small \[ \left\Vert \widehat{d}^{n}\left(
\tau_{-a}f\right)  \left( 0\right) \right\Vert _{\Theta
}=\left\Vert \widehat{d}^{n}f\left( a\right) \right\Vert _{\Theta
}\leq {\displaystyle\sum\limits_{j=0}^{\infty}}
\frac{1}{j!}\left\Vert\overset
{\begin{picture}(60,5)\put(0,0){\line(6,1){30}}\put(30,5){\line(6,-1){30}}\end{picture}}%
{d^{n+j}f\left(  0\right)  a^{j}}\right \Vert_{\Theta}
\leq {\displaystyle\sum\limits_{j=0}^{\infty}} \dfrac{n!\cdot
\sigma^{n+j}}{(n+j)!}\left\Vert \widehat{d}^{n+j}f\left( 0\right)
\right\Vert _{\Theta }\left\Vert a\right\Vert
^{j}%
\]}%
and%
{\small
\begin{gather*}
\left(  \frac{n}{ke}\right)  ^{\frac{n}{k}}\frac{1}{n!}\left\Vert
\widehat {d}^{n}\left(  \tau_{-a}f\right)  \left(  0\right)
\right\Vert _{\Theta  }\leq%
{\displaystyle\sum\limits_{j=0}^{\infty}}
\left(  \frac{n}{ke}\right) ^{\frac{n}{k}}\dfrac{
\sigma^{n+j}}{(n+j)!}\left\Vert
\widehat{d}^{n+j}f\left( 0\right)  \right\Vert _{\Theta  }\left\Vert a\right\Vert ^{j}\\
= {\displaystyle\sum\limits_{j=0}^{\infty}} \left(
\frac{n}{ke}\right)  ^{\frac{n}{k}}\left( \frac{ke}{n+j}\right)
^{\frac{n+j}{k}}\left(  \frac{n+j}%
{ke}\right)  ^{\frac{n+j}{k}}\frac{\sigma^{n+j}}{\left( n+j\right)
!}\left\Vert \widehat{d}^{n+j}f\left(  0\right) \right\Vert
_{\Theta  }\left\Vert a\right\Vert ^{j}\\
\leq {\displaystyle\sum\limits_{j=0}^{\infty}} \left(
\frac{ke}{j}\right)  ^{\frac{j}{k}}\sigma^{n+j}\left\Vert
a\right\Vert ^{j}\left(  \frac{n+j}{ke}\right)
^{\frac{n+j}{k}}\frac{1}{\left( n+j\right)  !}\left\Vert
\widehat{d}^{n+j}f\left(  0\right) \right\Vert _{\Theta  }.
\end{gather*}
}%
\newline Since $\lim\limits_{j\rightarrow\infty}\left(  \frac{ke}{j}\right)  ^{\frac{1}{k}%
}=0,$ for each $\varepsilon>0$ there is $D\left(  \varepsilon\right)  >0$ such that%
{\small
\begin{equation*}
\left(  \frac{ke}{j}\right)  ^{\frac{j}{k}}\leq D\left(
\varepsilon\right)
\varepsilon^{j}, \label{depsilon}%
\end{equation*}
}%
for all  $j\in\mathbb{N}$. Considering $\varepsilon>0$ such that
$\sigma\varepsilon^2\left\Vert a\right\Vert
 <1$ and using (\ref{epsilon+limsup}), we obtain
{\small
\[\left(  \frac{n}{ke}\right)  ^{\frac{n}{k}}\frac{1}{n!}\left\Vert
\widehat {d}^{n}\left(  \tau_{-a}f\right)  \left(  0\right)
\right\Vert _{\Theta}\leq C\left( \varepsilon\right)
D\left(  \varepsilon\right)  \sigma^{n}\varepsilon^{n}%
{\displaystyle\sum\limits_{j=0}^{\infty}}
\sigma^{j}\varepsilon^{2j}\left\Vert a\right\Vert
^{j}
=C\left(  \varepsilon\right)  D\left(  \varepsilon\right)
\sigma^{n}\varepsilon^{n}\frac{1}{1-\sigma\varepsilon
^2\left\Vert a\right\Vert}.
\]
}%
Hence {\small
\begin{equation*}
\limsup\limits_{n\rightarrow\infty}\left(  \frac{n}{ke}\right)  ^{\frac{1}{k}%
}\left\Vert \frac{\widehat{d}^{n}\left(  \tau_{-a}f\right)  \left(
0\right)
}{n!}\right\Vert _{\Theta  }%
^{\frac{1}{n}}=0, \label{limsuptrasl}%
\end{equation*}
}%
and so
$\tau_{-a}f\in Exp_{\Theta,0}^{k}\left(  E\right)$. To prove the convergence, let $f\in Exp_{\Theta
,0}^{k}\left(  E\right)  $ and $\rho>0$. Then

{\small
\begin{gather*}
\left\Vert \tau_{-a}f-%
{\displaystyle\sum\limits_{n=0}^{v}}
\frac{1}{n!}\widehat{d}^{n}f\left(  \cdot\right)  a\right\Vert
_{\Theta  ,k,\rho}\leq%
{\displaystyle\sum\limits_{j=0}^{\infty}} \rho^{-j}\left(
\frac{j}{ke}\right)  ^{\frac{j}{k}}
{\displaystyle\sum\limits_{n=v+1}^{\infty}}
\frac{1}{j!n!}\left\Vert \widehat{d}^{j}\left(
\widehat{d}^{n}f\left( \cdot\right)  a\right)  \left(  0\right)
\right\Vert _{\Theta  }\\
\leq%
{\displaystyle\sum\limits_{j=0}^{\infty}} \rho^{-j}\left(
\frac{j}{ke}\right)  ^{\frac{j}{k}}
{\displaystyle\sum\limits_{n=v+1}^{\infty}}
\frac{\sigma^{n+j}}{(n+j)!}\left\Vert \widehat{d}^{n+j}f\left(
0\right) \right\Vert _{\Theta }\left\Vert a\right\Vert
^{n}\\
={\displaystyle\sum\limits_{j=0}^{\infty}}
{\displaystyle\sum\limits_{n=v+1}^{\infty}} \rho^{-j}\left(
\frac{ke}{n+j}\right) ^{\frac{n}{k}}\sigma^{n+j}\left(
\frac{n+j}{ke}\right)  ^{\frac{n+j}{k}}\left\Vert
\frac{\widehat{d}^{n+j}f\left(  0\right)  }{\left(  n+j\right)
!}\right\Vert
_{\Theta  }\left\Vert a\right\Vert ^{n}\\
\leq C\left(  \varepsilon\right)  D\left(  \varepsilon\right)
{\displaystyle\sum\limits_{j=0}^{\infty}}
{\displaystyle\sum\limits_{n=v+1}^{\infty}}
\rho^{-j}\varepsilon^{n}\varepsilon^{n+j}\sigma^{n+j}\left\Vert a\right\Vert ^{n}\\
\leq C\left(  \varepsilon\right)  D\left(  \varepsilon\right)
{\displaystyle\sum\limits_{j=0}^{\infty}}\rho^{-j}\varepsilon^j\sigma^j
{\displaystyle\sum\limits_{n=v+1}^{\infty}}
\varepsilon^{2n}\sigma^{n}\left\Vert a\right\Vert ^{n}.
\end{gather*}
}

Now, if for each
$\rho>0$ we choose $\varepsilon>0$  such that
$\varepsilon < \frac{\rho}{\sigma}$ and $\varepsilon^ 2\sigma\left\Vert
a\right\Vert  <1,$ then
we obtain
{\small
\[
\lim\limits_{v\rightarrow\infty}\left\Vert \tau_{-a}f-%
{\displaystyle\sum\limits_{n=0}^{v}}
\frac{1}{n!}\widehat{d}^{n}f\left(  \cdot\right)  a\right\Vert
_{\Theta ,\rho}=0,
\]}%
and the convergence follows from the definition of the
topology.
\end{proof}

Using Proposition \ref{traslacao} it is not difficult to
show the following result.

\begin{proposition}
\label{limLambdatende a zero}Let $\left(  \mathcal{P}_{\Theta}(^{j}E)\right)
_{j=0}^{\infty}$ be a holomorphy type from $E$ to $\mathbb{C}$, $k\in\left[ 1,+\infty\right] $, $f\in Exp_{\Theta  ,0}^{k}\left( E\right)  $ and $a\in
E.$ Then {\small
\[
\lim_{\lambda\rightarrow0}\lambda^{-1}\left( \tau_{-\lambda
a}f-f\right)  =\widehat{d}^{1}f\left(  \cdot\right) a,
\]}%
in the topology of $Exp_{\Theta ,0}^{k}\left(  E\right)  .$
\end{proposition}

\begin{theorem}\label{teorema_translacao}Let $\left(  \mathcal{P}_{\Theta}(^{j}E)\right)
_{j=0}^{\infty}$ be a holomorphy type from $E$ to $\mathbb{C}$, $k\in\left[  1,+\infty\right]  $ and $L$ be a continuous linear mapping from 
$Exp_{\Theta,0}^{k}\left(
E\right)  $ into itself. Then $L$ is a convolution
operator if, and only if, $L\left( \tau_{a}f\right)
=\tau_{a}\left(  Lf\right)  $ for all $a\in E$ and
$f\in Exp_{\Theta
,0}^{k}\left(  E\right) .$
\end{theorem}

\begin{proof}
We saw that $L\left(  \widehat{d}^{n}f\left(
\cdot\right) a\right)  =\widehat{d}^{n}\left(  Lf\right)
\left(  \cdot\right) a$ for all $n\in\mathbb{N}$ and $a\in E.$ Using this fact and Proposition
\ref{traslacao} we have%
{\small
\[
L\left(  \tau_{-a}f\right)  =%
{\displaystyle\sum\limits_{n=0}^{\infty}}
\frac{1}{n!}L\left(  \widehat{d}^{n}f\left( \cdot\right)
\left(
a\right)  \right)  =%
{\displaystyle\sum\limits_{n=0}^{\infty}}
\frac{1}{n!}\widehat{d}^{n}\left(  Lf\right)  \left(
\cdot\right) a=\tau_{-a}\left(  Lf\right)  ,
\]}%
which implies $L\left(  \tau_{a}f\right) =\tau_{a}\left(
Lf\right)  .$ Now suppose that $L$ satisfies $L\left(  \tau_{a}f\right) =\tau_{a}\left(
Lf\right) $
for all $a\in E.$ Thus it follows from Proposition \ref{limLambdatende a zero} that%
{\small
\begin{gather*}
\widehat{d}^{1}\left(  Lf\right)  \left(  \cdot\right)
a=\lim_{\lambda\rightarrow0}\lambda^{-1}\left(  \tau_{-\lambda
a}\left(
Lf\right)  -Lf\right)  =\lim_{\lambda\rightarrow0}%
\lambda^{-1}\left(  L\left(  \tau_{-\lambda a}f\right)
-Lf\right) \\
=\lim_{\lambda\rightarrow0}L\left(  \lambda^{-1}\left(
\tau_{-\lambda a}f-f\right)  \right)  =L\left(
\lim_{\lambda \rightarrow0}\lambda^{-1}\left(  \tau_{-\lambda
a}f-f\right)  \right) =L\left(  \widehat{d}^{1}f\left(
\cdot\right)  a\right)  .
\end{gather*}
}%
Hence $L$ is a convolution operator.
\end{proof}

Now we are interested to provide a description of all convolution operators  on $Exp_{\Theta ,0}^{k}\left(
E\right)  .$ To do this, we need to introduce the convolution product.

\begin{definition}\rm
Let $\left(  \mathcal{P}_{\Theta}(^{j}E)\right)
_{j=0}^{\infty}$ be a holomorphy type from $E$ to $\mathbb{C}$, $k\in\left[
1,+\infty\right]  ,$ $T\in\left[ Exp_{\Theta  ,0}^{k}\left( E\right)  \right]  ^{\prime}$
and $f\in Exp_{\Theta
,0}^{k}\left(  E\right)  $. We define \emph{the convolution
product of }$T$ \emph{and} $f$ by $\left(  T\ast f\right)  \left(
x\right)  =T\left( \tau_{-x}f\right)  ,$ for all $x\in E.$
\end{definition}

We will prove that all
convolution operators are of the form $T\ast,$ but 
to prove that $T\ast$ defines a convolution operator on
$Exp_{\Theta ,0}^{k}\left(
E\right)  ,$ for $k\in\left[ 1,+\infty\right]  ,$ we need the following definition, which is a generalization of the concept of $\pi_2$-holomorphy type \cite[Definition 2.5]{favaro3}. The concept of  $\pi_2$-holomorphy type was used in \cite{favaro3} to describe convolution operators on $\mathcal{H}_{\Theta b}(E)$.

\begin{definition}\rm\label{pi2k}Let $k\in\left[  1,+\infty\right]  $ and $A\in\left[  0,+\infty\right)  $. A holomorphy type $(\mathcal{P}_{\Theta}(^{j}E))_{j=0}^{\infty}$ from $E$ to $\mathbb{C}$ is said to be a
\emph{$\pi_{2,k}$-holomorphy type} if,  for each $T\in\left[ Exp_{\Theta, 0,A}^{k}\left(  E\right) \right]  ^{\prime}$, the following conditions hold:

\begin{enumerate}
\item[(1)] For $j\in \mathbb{N}_{0}$ and $m \in \mathbb{N}_{0}$, $m\leq j,$ if
$P\in\mathcal{P}_{\Theta
}\left(  ^{j}E\right)  $ with $B\in\mathcal{L}\left(  ^{j}E\right)  $ such that
$P=\hat{B},$ then the
polynomial
{\small
\begin{align*}
T\left(  \widehat{B\cdot^{m}}\right)  \colon E  &  \longrightarrow\mathbb{C}\\
y  &  \longmapsto T\left(  B\cdot^{m}y^{j-m}\right)
\end{align*}
}%
belongs to $\mathcal{P}_{\Theta
}\left(  ^{j-m}E\right)$;
\item[(2)]For constants $C> 0$ and $\rho>A$ such that
{\small
\[
\left\vert T\left(  f\right)  \right\vert \leq C
\left\Vert f\right\Vert _{\Theta ,k,\rho}, \quad\text{if}\; k\in [1,+\infty),
\]}
{\small
\[
\left\vert T\left(  f\right)  \right\vert \leq Cp_{\Theta
,\rho}^{\infty}\left(  f\right)  ,\quad\text{if }k=+\infty,
\]}

for all $f\in Exp_{\Theta,0,A}^{k}\left(  E\right)$,  there is a constant $M > 0$ such that

{\small
\[
\left\Vert T\left(  \widehat{A\cdot^{m}}\right)  \right\Vert
_{\Theta  }\leq C M^ j\rho^{-m}\left(  \frac{m}{ke}\right)
^{\frac{m}{k}}\left\Vert P\right\Vert _{\Theta  },\quad\text{if}\; k\in [1,+\infty),
\]}
{\small
\[
\left\Vert T\left(  \widehat{A\cdot^{m}}\right)  \right\Vert
_{\Theta  }\leq CM^j\rho^{-m}\left\Vert P\right\Vert _{\Theta  },\quad\text{if}\; k=+\infty.
\]}
for every $P\in\mathcal{P}_{\Theta
}\left(  ^{j}E\right)  $,  $j\in \mathbb{N}_{0}$ and $m \in \mathbb{N}_{0}$, $m\leq j$.

\end{enumerate}

\end{definition}

\begin{remark}\rm
\begin{enumerate}
\item[(i)] Note that the constants $C$ and $\rho$ of Definition \ref{pi2k} (2) exist because $T\in\left[ Exp_{\Theta, 0,A}^{k}\left(  E\right) \right]  ^{\prime}$.

\item[(ii)] When $k=+\infty$ and $A=0$ the concepts of $\pi_{2,\infty}$-holomorphy type and $\pi_2$-holomorphy (see \cite[Definition 2.5]{favaro3}) type coincide. So in this case we write $\pi_{2,\infty}=\pi_{2}$.
\end{enumerate}
\end{remark}

\begin{theorem}\label{T*}Let $k\in\left[
1,+\infty\right] $ and $(\mathcal{P}_{\Theta}(^{j}E))_{j=0}^{\infty}$ be a $\pi_{2,k}$-holomorphy type from $E$ to $\mathbb{C}$. If 
$T\in\left[ Exp_{\Theta
,0}^{k}\left(  E\right) \right]  ^{\prime}$ and $f\in
Exp_{\Theta  ,0}^{k}\left(
E\right)  ,$ then $T\ast f\in Exp_{\Theta  ,0}^{k}\left(  E\right)  $ and $T\ast
\in\mathcal{A}_{\Theta  ,0}^{k}%
.$
\end{theorem}

\begin{proof}
The linearity of $T\ast$ is clear. By Propositions \ref{dnfa} and
\ref{traslacao} we have that%
{\small
\[
\left(  T\ast f\right)  \left(  x\right)  =T\left(  \tau_{-x}f\right)  =%
{\displaystyle\sum\limits_{n=0}^{\infty}}
\frac{1}{n!}T\left(  \widehat{d}^{n}f\left(  \cdot\right)  \left(
x\right)
\right)  =%
{\displaystyle\sum\limits_{n=0}^{\infty}}
\frac{1}{n!}%
{\displaystyle\sum\limits_{j=0}^{\infty}}
\frac{1}{j!}T\left(  \overset
{\begin{picture}(60,5)\put(0,0){\line(6,1){30}}\put(30,5){\line(6,-1){30}}\end{picture}}%
{d^{j+n}f\left(  0\right)  \cdot^{j}}\left(  x\right)  \right)  .
\]}%
Let $f\in
Exp_{\Theta  ,0}^{k}\left(
E\right)$ and $T\in\left[ Exp_{\Theta
,0}^{k}\left(  E\right) \right]  ^{\prime}$. Since $\Theta$ is a $\pi_{2,k}$-holomorphy type we have that $T\left(
\overset
{\begin{picture}(60,5)\put(0,0){\line(6,1){30}}\put(30,5){\line(6,-1){30}}\end{picture}}%
{d^{j+n}f\left(  0\right)  \cdot^{j}}\right)\in\mathcal{P}%
_{\Theta  }\left(
^{n}E\right)  $
and
{\small
\[
\left\Vert T\left(  \overset
{\begin{picture}(60,5)\put(0,0){\line(6,1){30}}\put(30,5){\line(6,-1){30}}\end{picture}}
{d^{j+n}f\left(  0\right)  \cdot^{j}}\right)  \right\Vert
_{\Theta  }\leq
CM^{j+n}\rho^{-j}\left(  \frac{j}{ke}\right) ^{\frac{j}{k}}\left\Vert
\widehat{d}^{j+n}f\left(  0\right)  \right\Vert _{\Theta  },
\]}
for  $k\in\left[ 1,+\infty\right)  ,$ and {\small
\[
\left\Vert T\left(  \overset
{\begin{picture}(60,5)\put(0,0){\line(6,1){30}}\put(30,5){\line(6,-1){30}}\end{picture}}%
{d^{j+n}f\left(  0\right)  \cdot^{j}}\right)  \right\Vert
_{\Theta  }\leq CM^{j+n}\rho^{-j}\left\Vert \widehat{d}%
^{j+n}f\left(  0\right)  \right\Vert _{\Theta  },
\]}%
for $k=+\infty,$ where $C$, $\rho$ and $M$ are as in Definition
\ref{pi2k}. If $k\in\left[
1,+\infty\right)  $ and $0<\rho^{\prime}<\rho,$ then

{\small
\begin{gather}%
{\displaystyle\sum\limits_{j=0}^{\infty}}
\frac{1}{j!}\left\Vert T\left(  \overset
{\begin{picture}(60,5)\put(0,0){\line(6,1){30}}\put(30,5){\line(6,-1){30}}\end{picture}}%
{d^{j+n}f\left(  0\right)  \cdot^{j}}\right)  \right\Vert
_{\Theta  }\leq C
{\displaystyle\sum\limits_{j=0}^{\infty}}
\frac{1}{j!}M^{j+n}(\rho^{\prime})^{-j}\left(  \frac{j}{ke}\right)
^{\frac{j}{k}}\left\Vert \widehat{d}^{j+n}f\left(  0\right)
\right\Vert _{\Theta  }\nonumber\\
=(\rho^{\prime})^{n}C  n!%
{\displaystyle\sum\limits_{j=0}^{\infty}}
M^{j+n}\frac{\left(  j+n\right)  !}{j!n!}\left(  \frac{j}{j+n}\right)  ^{\frac{j}{k}%
}\left(  \frac{ke}{j+n}\right)  ^{\frac{n}{k}}\left(
\frac{j+n}{ke}\right)
^{\frac{j+n}{k}}(\rho^{\prime})^{-\left(  j+n\right)  }\left\Vert \frac{\widehat{d}%
^{j+n}f\left(  0\right)  }{\left(  j+n\right)  !}\right\Vert
_{\Theta  }. \label{desig3.11}
\end{gather}
}%
Since
{\small
\[
\limsup\limits_{j\rightarrow\infty}\binom{j+n}{n}^{\frac{1}{j+n}}%
=1,
\]}%
then for every $\varepsilon>0$ there is $D\left(
\varepsilon\right)  >0$ such that%
{\small
\[
\binom{j+n}{n}\leq D\left(  \varepsilon\right)  \left(
1+\varepsilon\right) ^{j+n},
\]}%
for all $j\in\mathbb{N}$. Hence
{\small
\begin{gather*}
{\displaystyle\sum\limits_{j=0}^{\infty}}
\frac{1}{j!}\left\Vert T\left(  \overset
{\begin{picture}(60,5)\put(0,0){\line(6,1){30}}\put(30,5){\line(6,-1){30}}\end{picture}}%
{d^{j+n}f\left(  0\right)  \cdot^{j}}\right)  \right\Vert
_{\Theta  } 
\\
\leq C\cdot D\left(
\varepsilon\right)  (\rho^{\prime})^{n}n!\left(  \frac{ke}{n}\right)  ^{\frac{n}{k}
}\sum\limits_{j=0}^{\infty}\left(\frac{\rho^{\prime}}{M(1+\varepsilon)}\right)^{-(j+n)}\left(
\frac{j+n}{ke}\right)
^{\frac{j+n}{k}}\left\Vert \frac{\widehat{d}%
^{j+n}f\left(  0\right)  }{\left(  j+n\right)  !}\right\Vert
_{\Theta  }.%
\\
\leq C\cdot D\left(
\varepsilon\right)  (\rho^{\prime})^{n}n!\left(  \frac{ke}{n}\right)  ^{\frac{n}{k}%
}\left\Vert f\right\Vert _{\Theta
,k,\frac{\rho^{\prime}}{M(1+\varepsilon)}}%
\end{gather*}
}

and so
{\small
\[
P_{n}=%
{\displaystyle\sum\limits_{j=0}^{\infty}}
\frac{1}{j!}T\left(  \overset
{\begin{picture}(60,5)\put(0,0){\line(6,1){30}}\put(30,5){\line(6,-1){30}}\end{picture}}%
{d^{j+n}f\left(  0\right)  \cdot^{j}}\right)
\in\mathcal{P}_{\Theta
}\left(  ^{n}E\right)  ,
\]}%
for each $n\in\mathbb{N}$ and
{\small
\[
\left\Vert P_{n}\right\Vert _{\Theta }\leq C\cdot D\left(
\varepsilon\right)  (\rho^{\prime})^{n}n!\left( \frac{ke}{n}\right)
^{\frac{n}{k}}\left\Vert f\right\Vert _{\Theta  ,k,\frac{\rho^{\prime}}{M(1+\varepsilon)}}.
\]}

Hence%
{\small
\[
\limsup\limits_{n\rightarrow\infty}\left(  \frac{n}{ke}\right)  ^{\frac{1}{k}%
}\left\Vert \frac{P_{n}}{n!}\right\Vert _{\Theta
}^{\frac{1}{n}}\leq\limsup\limits_{n\rightarrow\infty
}(C\cdot D(\varepsilon))^{\frac{1}{n}}\rho^{\prime}\left\Vert f\right\Vert
_{\Theta  ,k,\frac{\rho^{\prime}}{M(1+\varepsilon)}}^{\frac{1}{n}}%
=\rho^{\prime}.
\]}%
Since $0<\rho^{\prime}<\rho$ is arbitrary we get
{\small
\[
\limsup\limits_{n\rightarrow\infty}\left(  \frac{n}{ke}\right)  ^{\frac{1}{k}%
}\left\Vert \frac{P_{n}}{n!}\right\Vert _{\Theta
}^{\frac{1}{n}}=0.
\]}
Thus $T\ast f\in
Exp_{\Theta  ,0}^{k}\left(
E\right)  .$ For $\rho_{1}>0,$
if we choose $0<\rho^{\prime}<\rho$ and $\rho^{\prime}<\rho_{1},$ then we have%
{\small
\begin{gather*}
\left\Vert T\ast f\right\Vert _{\Theta  ,k,\rho_{1}}=%
{\displaystyle\sum\limits_{n=0}^{\infty}}
\frac{1}{n!}\left(  \frac{n}{ke}\right)
^{\frac{1}{k}}\rho_{1}^{-n}\left\Vert
P_{n}\right\Vert _{\Theta  }\\
\leq C\left\Vert f\right\Vert _{\Theta  ,k,\frac{\rho^{\prime}}{M(1+\varepsilon)}}%
{\displaystyle\sum\limits_{n=0}^{\infty}}
\left(  \frac{\rho^{\prime}}{\rho_{1}}\right)  ^{n}=C\left(
1-\frac {\rho^{\prime}}{\rho_{1}}\right)  ^{-1}\left\Vert
f\right\Vert _{\Theta
,k,\frac{\rho^{\prime}}{M(1+\varepsilon)}}.
\end{gather*}
}%
and since the topology of $Exp_{\Theta  ,0}^{k}\left(
E\right)$ is generated by the family of norms $(\|\cdot \|_{\Theta,k,\rho})_{\rho>0}$, we have that $T\ast$ is continuous. The case $k=+\infty,$ was proved in
\cite[Proposition 4.7]{favaro3}.

The fact that $T*$ commutes with
translations is clear.
\end{proof}

\begin{definition}\rm\label{gamma_k}

If $k\in\left[
1,+\infty\right] $ and $(\mathcal{P}_{\Theta}(^{j}E))_{j=0}^{\infty}$ is a $\pi_{2,k}$-holomorphy type from $E$ to $\mathbb{C}$, we
define {\small
\[
\gamma_{\Theta  ,0}^{k}%
\colon\mathcal{A}_{\Theta  ,0}%
^{k}\longrightarrow\left[  Exp_{\Theta  ,0}^{k}\left(  E\right)  \right]  ^{\prime}
\]}%
by $\gamma_{\Theta ,0}^{k}\left( L\right)  \left(
f\right)  =\left( Lf\right)  \left( 0\right)  ,$ for
$f\in Exp_{\Theta ,0}^{k}\left( E\right)  $ and
$L\in\mathcal{A}_{\Theta ,0}^{k}.$
\end{definition}

\begin{theorem}
\label{bijLinear}If $k\in\left[
1,+\infty\right] $ and $(\mathcal{P}_{\Theta}(^{j}E))_{j=0}^{\infty}$ is a $\pi_{2,k}$-holomorphy type from $E$ to $\mathbb{C}$, then the mapping $\gamma_{\Theta  ,0}^{k}$ is a linear
bijection and its inverse is the mapping 
{\small
\[\Gamma_{\Theta
,0}^{k}\colon\left[ Exp_{\Theta  ,0}^{k}\left( E\right) \right]
^{\prime}\longrightarrow\mathcal{A}_{\Theta  ,0}^{k}
\]}%
given by $\Gamma_{\Theta  ,0}%
^{k}\left(  T\right)  \left(  f\right)  =T\ast f,$ for $T\in\left[
Exp_{\Theta ,0}^{k}\left(
E\right) \right]  ^{\prime},$ $f\in Exp_{\Theta  ,0}^{k}\left( E\right)  $ and $k\in\left[
1,+\infty\right]  .$

\end{theorem}
 
As a last result presented on this section we will prove that the Fourier-Borel transform becomes an isomorphism of algebras. We introduce the following product on  $\left[ Exp_{\Theta  ,0}^{k}\left(  E\right)
\right]  ^{\prime}$.

\begin{definition}
\rm\label{produtoConvolucao} Let $k\in\left[
1,+\infty\right] $, $(\mathcal{P}_{\Theta}(^{j}E))_{j=0}^{\infty}$ a $\pi_{2,k}$-holomorphy type from $E$ to $\mathbb{C}$ and $T_{1},T_{2}\in\left[ Exp_{\Theta ,0}^{k}\left( E\right)  \right]
^{\prime}$. We define the \emph{convolution product} $T_{1} \ast T_{2}\in \left[ Exp_{\Theta ,0}^{k}\left(  E\right)
\right]
^{\prime}$ by%
{\small
\[
T_{1}\ast T_{2}=\gamma_{\Theta ,0}^{k}\left(  L_{1}\circ L_{2}\right)
\in\left[ Exp_{\Theta
,0}^{k}\left(  E\right) \right]  ^{\prime},
\]}%
where $L_{1}=T_{1}\ast$ and
$L_{2}=T_{2}\ast.$
\end{definition}

It is easy to check that $\gamma_{\Theta  ,0}^{k}$ preserves product, that
is, $$\gamma_{\Theta ,0}^{k}\left(  L_{1}\circ L_{2}\right)
=\left( \gamma_{\Theta
,0}^{k}L_{1}\right)  \ast\left(  \gamma
_{\Theta  ,0}^{k}L%
_{2}\right)  .$$ With this product $\left[  Exp_{\Theta
,0}^{k}\left(  E\right)  \right]  ^{\prime}$ becomes an algebra
with unit element $\delta\colon Exp_{\Theta
,0}^{k}\left(  E\right) \rightarrow \mathbb{C}$ given by $\delta\left(  f\right)
=f\left(  0\right)  $, for all $f\in Exp_{\Theta
,0}^{k}\left(  E\right)$.

%
%
%

\begin{theorem}
\label{isomorfismoDeAlgebras}If $k\in\left(
1,+\infty\right] $ and $(\mathcal{P}_{\Theta}(^{j}E))_{j=0}^{\infty}$ is a $\pi_{1}$-$\pi_{2,k}$-holomorphy type from $E$ to $\mathbb{C}$, then the Fourier-Borel transform $\mathcal{F}$ is an 
isomorphism between the algebras $\left[ Exp_{\Theta  ,0}^{k}\left(  E\right)
\right]  ^{\prime}$ and $Exp_{\Theta^{\prime},\infty}^{k^{\prime}}\left( E^{\prime}\right)
=Exp_{\Theta^{\prime}}^{k^{\prime}}\left( E^{\prime}\right)
.$
\end{theorem}

\begin{proof}
Since $\Theta$ is a $\pi_{1}$-holomorphy type, then it follows from Theorem \ref{fourier_borel} that $\mathcal{F}$ is an algebraic
isomorphism between these spaces. Since $\Theta$ is also a $\pi_{2,k}$-holomorphy type, then it is easy to see that
$\mathcal{F}\left( T_{1}\ast T_{2}\right)  =\left( \mathcal{F}T_{1}\right)  \cdot\left(
\mathcal{F}T_{2}\right)  .$ 
\end{proof}

\section{Hypercyclicity results}

The main result of this section is the following theorem:

\begin{theorem}\label{main1} Let $k\in\left(
1,+\infty\right] $, $E^{\prime}$ be
separable and $(\mathcal{P}_{\Theta}(^{m}E))_{m=0}^{\infty}$ be a
$\pi_1$-holomorphy type from $E$ to $\mathbb{C}$. Then every nontrivial convolution operator
on $Exp_{\Theta  ,0}^{k}\left( E\right)$ is mixing and thus in particular hypercyclic.
\end{theorem}

The proof of this result rests on the following theorem, which, as mentioned in the Introduction, is due to Costakis and Sambarino \cite{costakis} and sharpens an earlier result of Kitai \cite{kitai} and Gethner and Shapiro \cite{GS}. 

\begin{theorem}
\label{Hypercyclity criterion}Let $X$ be a separable Fr\'echet space. Then a continuous linear 
mapping $T\colon X \rightarrow X$ is mixing if there are dense subsets $Z,Y$ of $X$ and a mapping
$S\colon Y \rightarrow Y$ satisfying the following three conditions:

\textit{(a) $T^{n}(z) \rightarrow 0$ when $n \rightarrow \infty$ for every $z \in Z$.}

\textit{(b) $S^{n}(y) \rightarrow 0$ when $n \rightarrow \infty$ for every $y \in Y$.}

\textit{(c) $T \circ S (y) = y$ for every $y \in Y$.}

\end{theorem}

Before proving Theorem \ref{main1} we need some auxiliary results.

\begin{proposition}\label{densidade de ephi b}
Let $k\in\left(1,+\infty\right] $, $(\mathcal{P}_{\Theta}(^{j}E))_{j=0}^{\infty}$ be a $\pi_{1}$-holomorphy type from $E$ to $\mathbb{C}$ and $U$ be a
non-empty open subset of $E^{\prime}$.  Then the set
$S={\rm span}\{e^{\phi}:\phi\in U \}$
is dense in $Exp_{\Theta  ,0}^{k}\left( E\right) $.
\end{proposition}

\begin{proof}
Assume that $S$ is not dense in $Exp_{\Theta  ,0}^{k}\left( E\right) $.  Since $Exp_{\Theta  ,0}^{k}\left(  E\right)  $ is a locally convex space, it follows as a consequence of  Hahn-Banach Theorem that there exists a nonzero functional $T\in [Exp_{\Theta  ,0}^{k}\left( E\right) ]^{\prime}$ that vanishes on $\overline{S}$.  In particular $T(e^{\phi})=0$ for each
$\phi\in U$ and so $\mathcal{F}T(\phi)=T(e^{\phi})=0$ for every $\phi \in U$
Thus $\mathcal{F} T$ is a holomorphic function that vanishes on the open non-empty set $U$ and this implies that $\mathcal{F} T\equiv 0$ on $E^\prime$. Since
$\mathcal{F}$ is injective we have $T\equiv 0,$ a contradiction. Hence $S$ is dense in $Exp_{\Theta  ,0}^{k}\left( E\right) $.
\end{proof}

Now we will prove that the exponential functions are eigenvectors for the convolution operators $L\in\mathcal{A}_{\Theta ,0}^{k}.$ Moreover, for each $L$ we will describe the eigenvalues associated to the exponential functions. This result is the key of the proof of Theorem \ref{main1}.

\begin{lemma}\label{lemma scalar}
Let $k\in\left( 1,+\infty\right] ,$ $(\mathcal{P}_{\Theta}(^{j}E))_{j=0}^{\infty}$ be a
$\pi_{1}$-holomorphy type from $E$ to $\mathbb{C}$ and
$L\in\mathcal{A}_{\Theta ,0}^{k}.$ Then:
\begin{enumerate}
\item[{\rm (a)}] $L(e^{\phi})=\mathcal{F}[\gamma_{\Theta
,0}^{k}\left( L\right)](\phi)\cdot e^{\phi }$ for every
$\phi\in E^{\prime}.$

\item[{\rm(b)}] $L$ is a scalar multiple of the identity if and only if
$\mathcal{F}\left[\gamma_{\Theta
,0}^{k}\left( L\right)\right]$ is constant.
\end{enumerate}
\end{lemma}

\begin{proof}
(a) By Definition \ref{gamma_k} and Theorem \ref{fourier_borel} we have $\gamma_{\Theta
,0}^{k}\left( L\right)\in\lbrack Exp_{\Theta  ,0}^{k}\left( E\right)]^{\prime}$ and
\[
\mathcal{F}[\gamma_{\Theta,0}^{k}(L)](\phi)=\gamma_{\Theta,0}^{k}(L)(e^{\phi
})=L(e^{\phi})(0)
\]
for each
$\phi\in E^{\prime}.$ Therefore%
\begin{align*}
L(e^{\phi})(y) &  =[\tau_{-y}(L(e^{\phi}))](0)
 =[L\left(  \tau_{-y}(e^{\phi})\right)  ](0) =[L\left(  e^{\phi(y)}\cdot e^{\phi}\right)  ](0)\\
&  =e^{\phi(y)}\cdot L(e^{\phi})(0) =e^{\phi(y)}\cdot \mathcal{F}[\gamma_{\Theta,0}^{k}(L)](\phi) = \left(\mathcal{F}[\gamma_{\Theta,0}^{k}(L)](\phi)
\cdot e^\phi \right)(y),
\end{align*}

for all $y\in E.$\newline(b) Let $\lambda\in\mathbb{C}$ be such that
$\mathcal{F}(\gamma_{\Theta,0}^{k}(L)(\phi)=\lambda$ for every $\phi\in
E^{\prime}.$ By (a) it follows that%
\[
L(e^{\phi})=\mathcal{F}[\gamma_{\Theta,0}^{k}(L)](\phi)\cdot
e^{\phi}=\lambda e^{\phi}
\]
for every $\phi\in
E^{\prime}.$ The continuity of $L$ and the denseness of $\{e^\phi : \phi \in E'\}$ in $Exp_{\Theta  ,0}^{k}\left( E\right)$ (Proposition \ref{densidade de ephi b}) yield
that $L(f)=\lambda f$ for every $f\in Exp_{\Theta  ,0}^{k}\left( E\right)$.

Conversely, let $\lambda\in\mathbb{C}$ be such that $L(f)=\lambda f,$ for every
$f\in Exp_{\Theta  ,0}^{k}\left( E\right)$. Using part (a) again we get%

\[
\lambda
e^{\phi}=L(e^{\phi})=\mathcal{F}[\gamma_{\Theta,0}^{k}
(L)](\phi)\cdot e^{\phi}
\]
and so $\mathcal{F}[\gamma_{\Theta,0}^{k}
(L)](\phi)=\lambda,$ for every $\phi\in E^{\prime}$. 
\end{proof}
\begin{proposition}\label{Linearmente indep}
Let $k\in\left(
1,+\infty\right] $ and $(\mathcal{P}_{\Theta}(^{j}E))_{j=0}^{\infty}$ be a $\pi_{1}$-holomorphy type from $E$ to $\mathbb{C}$. Then the set
$$B=\{e^\phi : \phi\in E^{\prime}\}$$

\noindent is a linearly independent subset of $Exp_{\Theta  ,0}^{k}\left( E\right) $.
\end{proposition}

\begin{proof} 
We know that $B \subseteq Exp_{\Theta  ,0}^{k}\left( E\right)$.
Let $\{e^{\phi_i}\}_{i\in I}$ be a maximal linearly independent subset of $B$. Fix $\phi\in E'$ and assume that there exist non-zero constants $ \alpha_{i_1}, \ldots, \alpha_{i_n}\in \mathbb{C}$ such that
\begin{equation}
\alpha_{i_1} e^{\phi_{i_1}} + \cdots + \alpha_{i_n} e^{\phi_{i_n}} = e^{\phi}\label{linear}
\end{equation}
Given $a \in E$, it follows from Proposition \ref{derivadas} that the
differentiation operator
$$D_a\colon Exp_{\Theta  ,0}^{k}\left( E\right)\longrightarrow Exp_{\Theta  ,0}^{k}\left( E\right)~,~D_a\left(
f\right) =\hat{d}^1f\left(  \cdot\right)  \left(  a\right)$$
is well defined. Applying the operator $D_{a}$ in $(\ref{linear})$, it follows that

\begin{equation}
\alpha_{i_1} \cdot \phi_{i_1}(a)\cdot e^{\phi_{i_1}} + \cdots + \alpha_{i_n}\cdot \phi_{i_n}(a)\cdot e^{\phi_{i_n}} = \phi(a)\cdot e^{\phi}\label{linear1}
\end{equation}

Since $\{e^{\phi_i}\}_{i\in I}$ is linearly independent and $ \alpha_{i_1}, \ldots, \alpha_{i_n}$ are non-zero, by $(\ref{linear})$ and $(\ref{linear1})$ we have
$$ \phi_{i_1}(a) = \cdots =  \phi_{i_n}(a)=\phi(a).$$
Since $a\in E$ is arbitrary we obtain
$$ \phi_{i_1} = \cdots =  \phi_{i_n}=\phi.$$
Hence $\{\phi_i\}_{i\in I}=E'$ and so $B$
is linearly independent.
\end{proof}

\bigskip

\textit{Proof of Theorem \ref{main1}.} Let $L \colon Exp_{\Theta  ,0}^{k}\left( E\right) \longrightarrow
Exp_{\Theta  ,0}^{k}\left( E\right)$ be a nontrivial convolution operator. We shall show that
$L$ satisfies the Hypercyclicity Criterion of Theorem
\ref{Hypercyclity criterion}. First of all, since $E'$ is
separable and $\Theta$ is a $\pi_1$-holomorphy type, we have that
$Exp_{\Theta  ,0}^{k}\left(  E\right) $ is separable as well.
 By \cite[Propositions 2.7]{favaro2}, $Exp_{\Theta  ,0}^{k}\left(  E\right)  $ is a Fr\'{e}chet space. By $\Delta$ we mean the open unit disk in $\mathbb{C}$. Consider the sets%

\[
V=\{\phi\in E^{\prime}: |\mathcal{F}
[\gamma_{\Theta,0}^{k}(L)](\phi)| <1
\}=\mathcal{F}[\gamma_{\Theta,0}^{k}(L)]^{-1}(\Delta)
\]

\noindent and%

\[
W=\{\phi\in E^{\prime}:
|\mathcal{F}[\gamma_{\Theta,0}^{k}(L)](\phi)| >1
\}=\mathcal{F}[\gamma_{\Theta,0}^{k}(L)]^{-1}(\mathbb{C}-\overline{\Delta}).
\]
Since $L$ is not a scalar multiple of the identity, Lemma
\ref{lemma scalar} (b) yields that
$\mathcal{F}[\gamma_{\Theta,0}^{k}(L)]$ is non constant.
Therefore, it follows from Liouville's Theorem that $V$ and $W$
are non-empty open subsets of $E^{\prime}$. Consider now the
following subspaces of  $Exp_{\Theta  ,0}^{k}\left( E\right)$:
\[
H_V = {\rm span}\{e^{\phi}: \phi\in V\} {\rm ~~and~~}H_W ={\rm
span}\{e^{\phi}: \phi\in W\}.
\]%
By Proposition \ref{densidade de ephi b} we know that both $H_V$
and $H_W$ are dense in $Exp_{\Theta  ,0}^{k}\left( E\right)$.

Let us deal with $H_V$ first. Given $\phi \in V$, from Lemma
\ref{lemma scalar}(a) we have
\[
L(e^{\phi})=\mathcal{F}[\gamma_{\Theta,0}^{k}(L)](\phi)\cdot
e^{\phi} \in H_V.
\]
So $L(H_V) \subseteq H_V$ because $L$ is linear. 
Applying Lemma \ref{lemma scalar}(a) and the linearity of $L$ once
again we get
\[
L^{n}(e^{\phi})=\left
(\mathcal{F}[\gamma_{\Theta,0}^{k}(L)](\phi)\right)^{n}\cdot
e^{\phi}
\]
for all $n\in\mathbb{N}$ and $\phi \in V$. Now let $f \in H_{V}$, that is $f = \sum\limits_{j=1}^{m} \alpha_{j} e^{\phi_{j}}$, with $\alpha_{j} \in \mathbb{C}$ and 
$\phi_{j} \in V$. It follows that 
$$ L^{n}(f) = \sum_{j=1}^{m} \alpha_j L^{n}(e^{\phi_{j}}) = \sum_{j=1}^{m} \alpha_{j} \left(\mathcal{F}[\gamma_{\Theta,0}^{k}(L)](\phi_j)\right)^{n}e^{\phi_{j}}. $$ 
Since $\left\vert\mathcal{F}[\gamma_{\Theta,0}^{k}(L)](\phi_j)\right\vert<1$ for every $j=1,\ldots,m$, it follows that $L^{n}f \rightarrow 0$ when $n \rightarrow \infty$. 

Now we handle $H_W$. For each $\phi\in W$,
$\mathcal{F}[\gamma_{\Theta,0}^k(L)](\phi) \neq 0$, so we can
define
\[
S(e^{\phi}):=\dfrac{e^{\phi}}{\mathcal{F}[\gamma_{\Theta,0}^k(L)](\phi)}
\in Exp_{\Theta  ,0}^{k}\left( E\right).
\]
By Proposition \ref{Linearmente indep}, $\{e^{\phi}: \phi\in W\}$
is a linearly independent set. Hence $S$ admits a unique extension
to a linear mapping $S\colon H_{W} \rightarrow H_{W}$. Now let $f \in H_{W}$, that is $f = \sum\limits_{j=1}^{m} \alpha_{j} e^{\phi_{j}}$,
with $\alpha_{j} \in \mathbb{C}$ and $\phi_{j} \in W$. It follows that
$$ S^{n}f = \sum_{j=1}^{m} \frac{ \alpha_{j} e^{\phi_{j}} }{\left
(\mathcal{F}[\gamma_{\Theta,0}^{k}(L)](\phi_j)\right)^{n}}. $$
Since $\left\vert\mathcal{F}[\gamma_{\Theta,0}^{k}(L)](\phi_j)\right\vert > 1$ for every $j$, it follows that $S^{n}f \rightarrow 0$ when $n \rightarrow \infty$.

Finally, $L\circ S(f)=f$ for every $f\in H_W$, so $L$ is
mixing.\hfill$\Box$

\begin{theorem}\label{main2}
Let $k\in\left( 1,+\infty\right] $, $E^{\prime}$ be separable,
$(\mathcal{P}_{\Theta}(^{m}E))_{m=0}^{\infty}$ be a
$\pi_1$-$\pi_{2,k}$-holomorphy type and $T\in[Exp_{\Theta
,0}^{k}\left( E\right)]^{\prime}$ be a linear functional which is
not a scalar multiple of $\delta_0$, where $\delta_0$ is defined by $\delta_0(f)=f(0)$. Then
$\Gamma_{\Theta,0}^{k}(T)$ is a convolution operator that is not a
scalar multiple of the identity, hence mixing.
\end{theorem}
\begin{proof}
By Theorem \ref{bijLinear}, for each $T\in[Exp_{\Theta ,0}^{k}\left( E\right)]^{\prime}$,  the mapping $\Gamma_{\Theta,0}^{k}(T)$ is a convolution operator. Suppose that there is
$\lambda\in\mathbb{C}$ such that
$\Gamma_{\Theta,0}^{k}(T)(f)=\lambda\cdot f$ for all
$f\in Exp_{\Theta ,0}^{k}\left( E\right)$. Then
\[
\lambda\cdot f(x)=\Gamma_{\Theta,0}^{k}(T)(f)(x)=(T\ast
f)(x)=T(\tau _{-x}f)
\]
for every $x \in E$. In particular,
\[
\lambda\cdot\delta_{0}(f)=\lambda\cdot f(0)=T(\tau_{0}f)=T(f)
\]
for every $f\in Exp_{\Theta ,0}^{k}\left( E\right)$. Hence
$T=\lambda\cdot\delta_{0}$. This contradiction shows that
$\Gamma_{\Theta,0}^{k}(T)$ is not a scalar multiple of the
identity, hence mixing by Theorem \ref{main1}.
\end{proof}

Now we have an easy but interesting application of Lemma \ref{lemma scalar}.

\begin{proposition}
\label{dense_range}Let $k\in\left( 1,+\infty\right] $, $E^{\prime}$ be separable and $(\mathcal{P}_{\Theta}(^{n}E))_{n=0}^{\infty
}$ be a $\pi_{1}$-holomorphy type from $E$ to $\mathbb{C}$. Then every nonzero convolution
operator on $Exp_{\Theta ,0}^{k}\left( E\right)$ has dense range.
\end{proposition}
\begin{proof}
Let $L\neq0$ be a convolution operator. If $L$ is a scalar multiple of the identity, then clearly $L$ is
surjective. Suppose now that $L$ is not a scalar multiple of the identity. By Proposition \ref{densidade de ephi b}, ${\rm span}\{e^{\phi}: \phi\in E^{\prime
}\}$ is dense in $Exp_{\Theta ,0}^{k}\left( E\right)$. By Lemma
\ref{lemma scalar} (a), $L(e^{\phi})=\mathcal{F}[\gamma_{\Theta,0}^k(L))(\phi)\cdot e^{\phi}$ for every $\phi \in E'$,  and this implies that each $e^{\phi}$ belongs to the range of $L$.
Therefore,
\[
Exp_{\Theta ,0}^{k}\left( E\right)=\overline{{\rm span}\{e^{\phi}: \phi\in E^{\prime}
\}}=\overline{L(Exp_{\Theta ,0}^{k}\left( E\right))}.%
\]
\end{proof}

\begin{remark} \label{canonical_cte}\rm
Note that, if $\sigma\neq 1$ and $A\neq 0$ then the concept of convolution operator on $Exp_{\Theta ,0,A}^{k}\left( E\right) $ is senseless because in this case Proposition \ref{derivadas} does not assure that $ d^{n}f\left(  \cdot\right)  a$ belongs to $Exp_{\Theta ,0,A}^{k}\left( E\right) $.  
In the most usual examples the inequality (2) of definition of holomorphy type (Definition \ref{holomorphy type}) works with the constant $\frac{j!}{k!(j-k)!}$ instead of $\sigma^{j}$.
Since $$\frac{j!}{k!(j-k)!}\leq 2^{j},$$ it follows that, in this case, inequality (2) of Definition \ref{holomorphy type} is valid with $\sigma=2$. So, if $\Theta$ is a holomorphy type such that condition (2) of Definition \ref{holomorphy type} is valid with $\frac{j!}{k!(j-k)!}$ instead of $\sigma^{j}$, for every $k,j\in\mathbb{N}_{0}$, $k\leq j $, then in Proposition \ref{derivadas} we obtain that $\widehat{d}^{n}f\left(  \cdot\right)  a\in Exp_{\Theta ,0, A}^{k}\left(  E\right)  ,$ instead of $\widehat{d}^{n}f\left(  \cdot\right)  a\in Exp_{\Theta ,0, \sigma A}^{k}\left(  E\right)  .$ Consequently, in this case, we also may define convolution operators from $Exp_{\Theta  ,0, A}^{k}\left(  E\right)$
to itself as in Definition \ref{defOperConv}.
We denote the set of all convolution
operators on $Exp_{\Theta  ,0,A}
^{k}\left(  E\right)  $ by  $\mathcal{A}_{\Theta
,0,A}^{k}.$ In this case  we say that $\Theta$ is a holomorphy type with \emph{canonical constants}.
\end{remark}

We finish this section stating the analogous of Theorem \ref{main1} in the case that  $\Theta$ is a holomorphy type with canonical constants, which encompasses most known examples. The proof is a straightforward adaptation of the previous results.


\begin{theorem}\label{main3} Let $k\in\left(
1,+\infty\right] $, $A\in\left[
0,+\infty\right)  $, $E^{\prime}$ be
separable and $(\mathcal{P}_{\Theta}(^{m}E))_{m=0}^{\infty}$ be a
$\pi_1$-holomorphy type from $E$ to $\mathbb{C}$, with canonical constants. Then every nontrivial convolution operator
on $Exp_{\Theta  ,0,A}^{k}\left( E\right)$ is mixing and thus in particular hypercyclic.
\end{theorem}


\section{Existence and Approximation Theorems for Convolution Equations}
\begin{definition}
\label{division space}\rm Let $U$ be an open subset of $E$ and
$\mathcal{F}(U)$ be a collection of holomorphic functions from $U$
into $\mathbb{C}$. We say that $\mathcal{F}(U)$ is \emph{closed
under division } if, for each $f$ and $g$ in $\mathcal{F}(U)$,
with $g\neq 0$ and $h=f/g$ a holomorphic function on $U$, we have
$h$ in $\mathcal{F}(U)$.\newline The quotient notation $h=f/g$
means that $f(x)=h(x)\cdot g(x)$, for all $x\in U$.
\end{definition}
The next useful result was proved by
Gupta in \cite{gupta}.
\begin{lemma}
\label{lema gupta} Let $U$ be an open connected subset of $E.$ Let
$f$ and $g$ be holomorphic functions on $U,$ with $g$ non
identically zero, such that, for any affine subspace $S$ of $E$ of
dimension one, and for any connected component $S^{\prime}$ of
$S\cap U$ on which $g$ is not identically zero, the restriction
$f|_{S^{\prime}}$ is divisible by the restriction
$g|_{S^{\prime}},$ with the quotient being holomorphic in
$S^{\prime}.$ Then $f$ is divisible by $g$ and the quotient is
holomorphic on $U.$
\end{lemma}

\begin{theorem}
\label{divisao} Let $k\in\left(
1,+\infty\right] $ and $(\mathcal{P}_{\Theta}(^{j}E))_{j=0}^{\infty}$ be a $\pi_{1}$-$\pi_{2,k}$-holomorphy type from $E$ to $\mathbb{C}$. If\linebreak
$Exp^{k'}_{\Theta^{\prime}}(E^{\prime})$ is closed under division and
$T_{1},T_{2}\in[Exp_{\Theta
,0}^{k}\left(  E\right)]^{\prime}$
are such that\ $T_{2}\neq0$ and $T_{1}\left(  P\cdot e^\phi\right)
=0$ whenever $T_{2}\ast( P\cdot e^\phi)=0$ with $\phi\in E^{\prime}$ and
$P\in\mathcal{P}_{\Theta}\left(  ^{m}E\right)  ,$
$m\in\mathbb{N}_{0},$ then $\mathcal{F}T_{1}$ is divisible by
$\mathcal{F}T_{2}$ with the quotient being an element of
$Exp^{k'}_{\Theta^{\prime}}(E^{\prime}).$
\end{theorem}

\begin{proof}
 Let $S$ be an one dimensional affine subspace of
$E^{\prime}.$ It is clear that $S$ is of the form $\left\{
\phi_{1}+t\phi_{2};t\in\mathbb{C}\right\} ,$ where
$\phi_{1},\phi_{2}\in E^{\prime}$ are fixed. We suppose that
$t_{0}$
is a zero of order $k$ of%
\[
g_{2}\left(  t\right)  =\mathcal{F}T_{2}  \left(
\phi _{1}+t\phi_{2}\right)  =T_{2}\left(  e^{
\phi_{1}+t\phi_{2}} \right)  .
\]
Then we have%
\[
T_{2}\left(  \phi_{2}^{j}\cdot e^{ \phi_{1}+t_{0}\phi_{2}}
\right) =0,
\]
for each $j<k,$ and this implies%
$$
T_{2}\ast\left(\phi_{2}^{j}\cdot e^{ \phi_{1}+t_{0}\phi_{2}}\right) =%
{\displaystyle\sum\limits_{m=0}^{j}}
\binom{j}{m}\phi_{2}^{j-m}\cdot e^{ \phi_{1}+t_{0}\phi_{2}}\cdot
T_{2}\left(  \phi_{2}^{m}\cdot e^{ \phi_{1}+t_{0}\phi_{2}}
\right) =0,
$$ for each $j<k.$ Hence, it follows from the hypothesis that
$T_{1}\left( \phi_{2}^{j}\cdot e^{ \phi_{1}+t_{0}\phi_{2}}
\right) =0,$ for all $j<k,$ and this implies that $t_{0}$ is a
zero of order at least $k$ of $g_{1}\left(  t\right)
=\mathcal{F}T_{1}\left( \phi_{1}+t\phi_{2}\right)  .$ Therefore
$\mathcal{F}T_{1}|_{S}$ is divisible by $\mathcal{F}T_{2}|_{S}$
and the quotient is holomorphic on $S.$ By Lemma \ref{lema gupta},
we have that $\mathcal{F}T_{1}$ is divisible by $\mathcal{F}T_{2}$
on $E^{\prime}$ and the quotient is an entire function. Since $Exp^{k'}_{\Theta^{\prime}}(E^{\prime})$ is closed under division, then the quotient $\mathcal{F}T_{1}/\mathcal{F}T_{2}$ belongs to
$Exp^{k'}_{\Theta^{\prime}}(E^{\prime}).$
\end{proof}

\begin{theorem}
\label{teoremaAproximacao1}Let $k\in\left(
1,+\infty\right] $ and $(\mathcal{P}_{\Theta}(^{j}E))_{j=0}^{\infty}$  be a $\pi_{1}$-$\pi_{2,k}$-holomorphy type from $E$ to $\mathbb{C}$. If  $Exp^{k'}_{\Theta^{\prime}}(E^{\prime})$ is closed under division and $L\in\mathcal{A}_{\Theta  ,0}^{k},$ then the vector
subspace of $Exp_{\Theta  ,0}%
^{k}\left(  E\right)  $ ge\-ne\-ra\-ted by%
{\small
\[ \mathcal{L=}\left\{
P\cdot e^\varphi\colon P\in\mathcal{P}_{\Theta  }\left(  ^{n}E\right)
,n\in\mathbb{N}_0,\varphi\in E^{\prime} \textrm{ and }L\left(
P\cdot e^\varphi\right)  =0\right\}
\]}%
is dense in%
{\small
\[ \ker L=\left\{  f\in Exp_{\Theta ,0}^{k}\left(  E\right)
\colon Lf=0\right\}  .
\]}%
\end{theorem}

\begin{proof}
First let us consider $L$ equal to $0.$ In this case $\ker L= Exp_{\Theta ,0}^{k}\left(  E\right)$ and  the result follows from Proposition \ref{densidade de ephi b}.  
Now consider
$L\neq 0$. By Theorem \ref{bijLinear} there is $T\in [Exp_{\Theta
,0}^{k}\left(  E\right)  ]'$ such
that $L=T\ast.$ Suppose that $R\in [Exp_{\Theta
,0}^{k}\left(  E\right)  ]'$ is such that $R|_{\mathcal{L}}=0.$
Thus by Theorem \ref{divisao}, there is $H\in Exp_{\Theta^{\prime}}^{k'}\left( E^{\prime}\right) $ such that
$\mathcal{F}\left(  R\right) =H\cdot\mathcal{F}\left(  T\right) .$
By the isomorphism of the Fourier-Borel transform (see Theorem \ref{fourier_borel}) there is $S\in[Exp_{\Theta
,0}^{k}\left(  E\right)  ]'$ such
that $H=\mathcal{F}\left( S\right)  $ and $\mathcal{F}\left(
R\right)  =\mathcal{F}\left( S\right)\cdot \mathcal{F}\left(
T\right) =\mathcal{F}\left(  S\ast T\right)  .$ Hence $R=S\ast T$
and for each $f\in {\rm ker }L,$ we have $R\ast f=S\ast\left(
T\ast f\right)=S\ast (Lf) =0$ and $R\left( f\right)  =\left(  R\ast f\right)
\left( 0\right)  =0.$ We showed that every $R\in [Exp_{\Theta
,0}^{k}\left(  E\right)  ]'$
vanishing on the vector subspace of $Exp_{\Theta
,0}^{k}\left(  E\right)  $ generated by $\mathcal{L}$ vanishes on
$\ker L$. Now the result follows from the
Hahn-Banach Theorem.
\end{proof}

\begin{theorem}
\label{teorema Ortogonal e fraco*}Let $k\in\left(
1,+\infty\right] $ and $(\mathcal{P}_{\Theta}(^{j}E))_{j=0}^{\infty}$ be a $\pi_{1}$-$\pi_{2,k}$-holomorphy type from $E$ to $\mathbb{C}$. If $Exp^{k'}_{\Theta^{\prime}}(E^{\prime})$ is closed under division and $L\in\mathcal{A}_{\Theta  ,0}^{k}$, then
its transpose \newline$^{t}L\colon\left[
Exp_{\Theta  ,0}^{k}\left(
E\right)  \right]  ^{\prime}\longrightarrow\left[
Exp_{\Theta ,0}^{k}\left(
E\right) \right]  ^{\prime}$is such that\newline\emph{(a)
}$^{t}L\left( \left[ Exp_{\Theta ,0}^{k}\left( E\right)  \right]
^{\prime}\right)  $ is the orthogonal of $\ker L$
in $\left[  Exp_{\Theta  ,0}%
^{k}\left(  E\right)  \right]  ^{\prime}.$\newline\emph{(b) }$^{t}%
L\left(  \left[  Exp_{\Theta  ,0}^{k}\left(  E\right)  \right]
^{\prime}\right)  $ is closed for the weak-star topology
in $\left[  Exp_{\Theta ,0}^{k}\left(  E\right)  \right]  ^{\prime}$ defined by

$Exp_{\Theta,0}^{k}\left( E\right)  .$
\end{theorem}
\begin{proof} If $L$ is equal to $0,$ the result is clear. Let
$L\neq 0$ and $T\in \left[ Exp_{\Theta ,0}^{k}\left( E\right)  \right]
^{\prime}$ be such
that $L=T\ast.$ For each $R\in \hspace{0.1 cm}^{t}L\left( \left[ Exp_{\Theta ,0}^{k}\left( E\right)  \right]
^{\prime}\right)  $ there is $S\in  \left[ Exp_{\Theta ,0}^{k}\left( E\right)  \right]
^{\prime}$ satisfying
$R=\hspace{0.1 cm}^{t}L\left( S\right)  .$ Hence, for each $f\in \ker L$ we have $R\left( f\right) =\hspace{0.1 cm}^{t}L\left(
S\right)  \left(  f\right)  =S\left( Lf\right)  =0,$ and
then $^{t}L\left( \left[ Exp_{\Theta ,0}^{k}\left( E\right)  \right]
^{\prime}\right)  $ is contained in the
orthogonal of $\ker L.$ Conversely, if $R$ is in the
orthogonal of $\ker L,$ then by Theorem \ref{divisao} there
is $H\in Exp_{\Theta^{\prime}}^{k'}\left( E^{\prime }\right)  $ such
that $\mathcal{F}\left(  R\right) =H\cdot\mathcal{F}\left(
T\right)  $ and by Theorem \ref{fourier_borel} there is $S\in  \left[ Exp_{\Theta ,0}^{k}\left( E\right)  \right]
^{\prime}$ such that $H=\mathcal{F}\left( S\right)  $ and
$\mathcal{F}\left( R\right)  =\mathcal{F}\left(
S\right)\cdot\mathcal{F}\left( T\right) =\mathcal{F}\left(  S\ast
T\right)  .$ Hence $R=S\ast T$ and for each $f\in Exp_{\Theta ,0}^{k}\left( E\right) ,$ we
have%
\begin{gather*}
R\left(  f\right)  =\left(  S\ast T\right)  \left(  f\right)
=\left(  \left( S\ast T\right)  \ast f\right)  \left(  0\right)
=\left(  S\ast\left(  T\ast
f\right)  \right)  \left(  0\right)  \\
=S\left(  T\ast f\right)  =S\left(  Lf\right)  =\hspace{0.1cm}^{t}L%
\left(  S\right)  \left(  f\right)
\end{gather*}
and this implies that $R=\hspace{0.1cm}^tL\left(
S\right)  $ and so $R\in\hspace{0.1cm}^{t}L\left( \left[ Exp_{\Theta ,0}^{k}\left( E\right)  \right]
^{\prime}\right)  $, proving $\left(  a\right)  $. To prove $(b)$ note
that the orthogonal of $\ker L$ is equal
to%
\[%
{\displaystyle\bigcap\limits_{f\in\ker L}}
\left\{  T\in\left[ Exp_{\Theta ,0}^{k}\left( E\right)  \right]
^{\prime};\text{ }T\left(  f\right)  =0\right\}  .
\]
Since, for each $f\in Exp_{\Theta ,0}^{k}\left( E\right)$
the set $\left\{  T\in\left[  \mathcal{H}_{\Theta b}\left(
E\right)  \right] ^{\prime}\colon T\left(  f\right)  =0\right\}
$ is closed for the weak-star topology, the result follows.
\end{proof}

\qquad The next result of this article is a theorem about
existence of solution of convolution equations. In order to prove
this result we need the following Dieudonn\'{e}-Schwartz
result (see \cite[Th\'eor\`eme 7]{DS} or \cite[p. 308]{horvath}).

\begin{lemma}
\label{lema teorema de existencia}If $E$ and $F$ are Fr\'{e}chet
spaces and $u\colon E\longrightarrow F$ is a linear continuous
mapping, then the following conditions are
equivalent:\newline\emph{(a) }$u\left(  E\right)
=F;$\newline\emph{(b) }$^{t}u\colon F^{\prime}\longrightarrow
E^{\prime}$ is injective and $^{t}u\left(  F^{\prime}\right)  $ is
closed for the weak-star topology of $E^{\prime}$ defined by $E.$
\end{lemma}

\begin{theorem}
\label{teorema de existencia}Let $k\in\left(
1,+\infty\right] $ and $(\mathcal{P}_{\Theta}(^{j}E))_{j=0}^{\infty}$ be a $\pi_{1}$-$\pi_{2,k}$-holomorphy type from $E$ to $\mathbb{C}$. If  $Exp^{k'}_{\Theta^{\prime}}(E^{\prime})$ is closed under division and $L\in\mathcal{A}_{\Theta  ,0}^{k},$ then $L\left(
Exp_{\Theta  ,0}^{k}\left(E\right)  \right) =Exp_{\Theta  ,0}^{k}\left( E\right)  .$
\end{theorem}

\begin{proof}
By \cite[Proposition 2.7]{favaro2}, $Exp_{\Theta  ,0}^{k}\left(  E\right)  $ is a Fr\'{e}chet space. By Lemma
\ref{lema teorema de existencia} (b) and Theorem \ref{teorema
Ortogonal e fraco*} $(b)$, it is enough to show
that $^{t}L$ is injective.  Since $L=T\ast$ for some $T\in \left[
\mathcal{H}_{\Theta b}\left(  E\right)  \right]  ^{\prime}$ then,
for all $S\in \left[ Exp_{\Theta  ,0}^{k}\left(  E\right)
\right]  ^{\prime }$ and $f\in Exp_{\Theta  ,0}^{k}\left(  E\right)   $ we have $$
^tL (S)  \left(  f\right)  =S\left(  L%
f\right)  =S\left(  T\ast f\right)  =\left(  S\ast T\right) \left(
f\right)
.$$ Thus $^tL (S)=S\ast T$ and if $^tL(S)=0,$ then $S\ast T=0$ and $\mathcal{F}\left(  S\ast T\right)
=0.$ Since $L=T\ast$ is non zero, then it follows that
$\mathcal{F}T$ is non zero and since $\mathcal{F}\left( S\ast
T\right) =\mathcal{F}S\cdot\mathcal{F}T,$ we get $\mathcal{F}S=0.$
Hence $S=0$ and $^tL$ is injective.
\end{proof}

%
%

\section{Applications}

We finish the paper showing the applicability of Theorems \ref{main1}, \ref{main2}, \ref{main3},  \ref{teoremaAproximacao1} and \ref{teorema de existencia}.

\bigskip

\noindent (1) Consider the finite dimensional case $E=\mathbb{C}^n.$ Then, $\mathcal{P}_{\Theta}(^{m}E)=\mathcal{P}(^{m}\mathbb{C}^{n}),$ for all $m\in\mathbb{N}_0$. Considering $k=\infty$ we get $Exp_{\Theta  ,0}^{k}\left( E\right)=\mathcal{H}(\mathbb{C}^{n})$. Thus, using Theorem \ref{main1} we recover the result of Godefroy
and Shapiro \cite{godefroy} (and consequently the results of Birkhoff \cite{birkhoff} and MacLane \cite{maclane}) which states that every nontrivial convolution operator 
$L\colon\mathcal{H}(\mathbb{C}^{n})\rightarrow \mathcal{H}(\mathbb{C}^{n})$ is hypercyclic. Moreover, for an arbitrary $k\in(1,+\infty]$ we obtain the unknown result that every nontrivial convolution operator 
$L\colon Exp_{0}^{k}(\mathbb{C}^{n})\rightarrow Exp_{0}^{k}(\mathbb{C}^{n})$ is mixing. More generally, since $(\mathcal{P}(^{m}\mathbb{C}^{n}))_{m=0}^\infty$ is a holomorphy type with canonical constants according to Remark \ref{canonical_cte}, then, for $A\in\left[  0,+\infty\right)  $, we also obtain that every nontrivial convolution operator $L\colon Exp_{0, A}^{k}(\mathbb{C}^{n})\rightarrow Exp_{0, A}^{k}(\mathbb{C}^{n})$ is mixing. 
It is easy to check that $(\mathcal{P}(^{m}\mathbb{C}^{n}))_{m=0}^\infty$ is a $\pi_{2,k}$-holomorphy type. Hence it follows from Theorem \ref{main2} that if $T\in[Exp_{0}^{k}\left( \mathbb{C}^{n}\right)]^{\prime}$ is a linear functional which is
not a scalar multiple of $\delta_0$, then the convolution operator
$\Gamma_{\Theta,0}^{k}(T)$ on $Exp_{0}^{k}(\mathbb{C}^{n})$ is mixing.

Finally, using \cite[Corollaire 1]{martineau} it is easy to check that $Exp^{k'}_{\Theta^{\prime}}(\mathbb{C}^{n})=Exp^{k'}(\mathbb{C}^{n})$ is closed under division. Therefore it follows from Theorem \ref{teorema de existencia} that for each $g\in Exp_{0}^{k}(\mathbb{C}^{n}),$ the convolution equation $Lf=g$ has a solution $f\in Exp_{0}^{k}(\mathbb{C}^{n})$. Moreover, it follows from Theorem \ref{teoremaAproximacao1} that each solution of the homogeneous equation $Lf=0$ can be approximated by exponential polynomials solutions. Hence we recover the existence and approximation results for convolution operators on $Exp_{0}^{k}(\mathbb{C}^{n})$ obtained by Martineau \cite{martineau}.

\bigskip

\noindent (2) Let $E$ be a complex Banach space such that $E^\prime$ has the bounded approximation property and consider the space $\mathcal{P}_{N}(^{m}E)$ of all nuclear $m$-homogeneous polynomials on $E$. It is well known that $\Theta=N$  is a $\pi_1$-holomorphy type (see for instance Dineen \cite[Example 3]{dineen}). It is also easy to check that $\Theta=N$ is a holomorphy type with canonical constants. By Matos \cite[Proposition 3.9]{matos3} $\Theta=N$ is also a $\pi_{2,k}$-holomorphy type. Since $\Theta^\prime$ is the holomorphy type of all continuous $m$-homogeneous polynomials on $E^\prime$ we have that $Exp^{k'}_{\Theta^{\prime}}(E^\prime)=Exp^{k'}(E^\prime)$
and by Matos \cite[Theorem 4.3]{matos2} 
$$
\left[Exp^{k}_{N,0}(E)\right]^{\prime}=Exp^{k'}(E^\prime).
$$
Moreover, since $$Exp_0^{k'}(E^\prime)\subset Exp_A^{k'}(E^\prime)\subset Exp_{0,A}^{k'}(E^\prime)\subset Exp_B^{k'}(E^\prime)\subset Exp^{k'}(E^\prime),$$ for every $k\in[1,\infty]$ and $0<A<B<\infty$, then using Matos \cite[Corollary 4.5]{matos3} it is easy to check that $Exp^{k'}(E^\prime)$ is closed under division. Thus, we have the following unknown results:
\begin{itemize}

\item Every nontrivial convolution operator $L\colon Exp_{N,0, A}^{k}(E)\rightarrow Exp_{N,0, A}^{k}(E)$ is mixing, for every $k\in(1,+\infty]$ and $A\in\left[  0,+\infty\right)  $.

\item If $T\in[Exp_{N,0}^{k}\left( E\right)]^{\prime}$ is a linear functional which is
not a scalar multiple of $\delta_0$, then the convolution operator
$\Gamma_{N,0}^{k}(T)$ on $Exp_{N,0}^{k}(E^{\prime})$ is mixing.

\end{itemize}

If $k\in(1,+\infty]$ and $L\colon Exp_{N,0}^{k}(E)\rightarrow Exp_{N,0}^{k}(E)$ is a convolution operator, then we also recover the following results of  \cite{matos3}:

\begin{itemize}

\item For each $g\in Exp_{N,0}^{k}(E),$ the convolution equation $Lf=g$ has a solution $f\in Exp_{N,0}^{k}(E)$. 

\item Each solution of the homogeneous equation $Lf=0$ can be approximated by exponential polynomials solutions in $Exp_{N,0}^{k}(E).$

\end{itemize}

\bigskip

\noindent (3) Let $0<s\leq\infty$ and $1\leq q,r\leq\infty$ such that $q^\prime\leq r^\prime$ and $$1\leq\frac{1}{s}+\frac{m}{q^\prime},$$ for every $m\in\mathbb{N}$ (as usual $s^\prime, r^\prime, q^\prime$ denote the conjugates of $s, r, q$, respectively). Consider the space $\mathcal{P}_{\widetilde{N},\left( s;\left( r,q\right)
\right)}\left(^{m}E\right)$ of all $\left(
s;\left( r,q\right) \right)$-quasi-nuclear $m$-homogeneous
polynomials on the complex Banach space $E$ introduced by Matos \cite[Section 7.2]{Matos-livro}. Let us consider also $E^\prime$ having the bounded approximation property. The proof that $\left(\mathcal{P}_{\widetilde{N},\left( s;\left( r,q\right)
\right)}\left(^{m}E\right)\right)_{m=0}^\infty$ is a $\pi_1$-holomorphy type can be found  in \cite[Sections 8.2 and 8.3]{Matos-livro} (see also \cite[Example 3.11. (a)]{favaro3}) and it is easy to check that this is a holomorphy type with canonical constants. By F\'avaro \cite[Proposition 2.14]{FaBelg} $\left(\mathcal{P}_{\widetilde{N},\left( s;\left( r,q\right)
\right)}\left(^{m}E\right)\right)_{m=0}^\infty$ is a $\pi_{2,k}$-holomorphy type. Moreover, Matos proved in \cite[Section 8.2]{Matos-livro} that when
$E^{\prime}$ has the bounded approximation property,
then the Borel transform $\mathcal{B}_{\widetilde{N},\left(
s;\left( r,q\right) \right) }$ establishes an isometric
isomorphism between $\left[\mathcal{P}_{\widetilde{N},\left(
s;\left( r,q\right) \right)}\left(^{m}E\right)\right]^{\prime}$
and $\mathcal{P}_{\left( s^{\prime},m\left(
r^{\prime};q^{\prime}\right) \right)}\left(^{m}E^{\prime}\right)$, where $\mathcal{P}_{\left( s^{\prime},m\left(
r^{\prime};q^{\prime}\right) \right)}\left(^{m}E^{\prime}\right)$ denotes the space of all absolutely $\left( s^{\prime},m\left(
r^{\prime};q^{\prime}\right) \right)$-summing $m$-homogeneous
polynomials on $E^\prime$ introduced by Matos \cite[Section 3]{matosjmaa}. So, in this case the role of $Exp^{k'}_{\Theta^{\prime}}(E^{\prime})$ is played by $Exp^{k'}_{\left( s^{\prime},m\left(
r^{\prime};q^{\prime}\right) \right)}(E^{\prime})$ and the isomophism 
$$
\left[Exp^{k}_{\widetilde{N},\left(
s;\left( r,q\right) \right) }(E)\right]^{\prime}=Exp^{k'}_{\left( s^{\prime},m\left(
r^{\prime};q^{\prime}\right) \right)}(E^\prime)
$$
given by the Fourier-Borel transform is in F\'avaro \cite[Theorem 3.5]{favaro}.
 
Furthermore, since $$Exp_{0,\Theta^\prime}^{k'}(E^\prime)\subset Exp_{\Theta^\prime,A}^{k'}(E^\prime)\subset Exp_{\Theta^\prime,0,A}^{k'}(E^\prime)\subset Exp_{\Theta^\prime,B}^{k'}(E^\prime)\subset Exp^{k'}_{\Theta^\prime}(E^\prime),$$ for every $k\in[1,\infty]$ and $0<A<B<\infty$, then using F\'avaro \cite[Theorem 3.5 and Remark 3.6]{FaBelg} it is easy to check that $Exp^{k'}_{\left( s^{\prime},m\left(
r^{\prime};q^{\prime}\right) \right)}(E^\prime)$ is closed under division. Thus, we have the following unknown results:

\begin{itemize}

\item Every nontrivial convolution operator $L\colon Exp_{\widetilde{N},\left(
s;\left( r,q\right) \right),0, A}^{k}(E)\rightarrow Exp_{\widetilde{N},\left(
s;\left( r,q\right) \right),0, A}^{k}(E)$ is mixing, for every $k\in(1,+\infty]$ and $A\in\left[  0,+\infty\right)  $.

\item If $T\in[Exp_{\widetilde{N},\left(
s;\left( r,q\right) \right),0}^{k}\left( E\right)]^{\prime}$ is a linear functional which is
not a scalar multiple of $\delta_0$, then the convolution operator
$\Gamma_{\widetilde{N},\left(
s;\left( r,q\right) \right),0}^{k}(T)$ on $Exp_{\widetilde{N},\left(
s;\left( r,q\right) \right),0}^{k}(E^{\prime})$ is mixing.

\end{itemize}

If $k\in(1,+\infty]$ and $L\colon Exp_{\widetilde{N},\left(
s;\left( r,q\right) \right),0}^{k}(E)\rightarrow Exp_{\widetilde{N},\left(
s;\left( r,q\right) \right),0}^{k}(E)$ is a convolution operator, then we also recover the following results of  \cite{FaBelg}:

\begin{itemize}

\item For each $g\in Exp_{\widetilde{N},\left(s;\left( r,q\right) \right),0}^{k}(E),$ the convolution equation $Lf=g$ has a solution $f\in Exp_{\widetilde{N},\left(
s;\left( r,q\right) \right),0}^{k}(E)$. 

\item Each solution of the homogeneous equation $Lf=0$ can be approximated by exponential polynomials solutions in $Exp_{\widetilde{N},\left(
s;\left( r,q\right) \right),0}^{k}(E).$

\end{itemize}

\bigskip

\noindent (4) We can also obtain the hypercyclicity result given in Theorem \ref{main3} for the following holomorphy types (both are holomorphy types with canonical constants):

\begin{itemize}

\item $\left(  \mathcal{P}_{\widetilde{N},\left(  \left(
r,q\right)  ;\left(  s,p\right)  \right)  }\left(  ^{m}E\right)  \right)
_{m=0}^{\infty}:$ the holomorphy type of all \emph{Lorentz $((r,q);(s,p))$
-quasi-nuclear} $m$-homogeneous polynomials from $E$ to $\mathbb{C}$, $m\in\mathbb{C},$ where  $r,q,s,p\in\lbrack1,\infty\lbrack$,
$r\leq q$, $s^{\prime}\leq p^{\prime}$ and%
\[
1\leq\frac{1}{q}+\frac{m}{p^{\prime}}, \textrm{ for all } m\in\mathbb{C}.
\]
See \cite[Section 2 and Definition 4.4]{FP}.  

\item $\left(\mathcal{P}_{\sigma(p)}\left(^{m}E\right)\right)
_{m=0}^{\infty}:$ the holomorphy type of all $\sigma(p)$\emph{-nuclear} $m$-homogeneous polynomials from $E$ to $\mathbb{C}$, $m\in\mathbb{C}$, where $p\geq 1,$ defined in the obvious way according to the multilinear case studied in \cite{BM}.

\end{itemize}
Consider $E^\prime$ having the bounded approximation property. Then \cite[Propositions 4.9]{FP} and \cite[p.7]{BM} ensure that $\left(  \mathcal{P}_{\widetilde{N},\left(  \left(
r,q\right)  ;\left(  s,p\right)  \right)  }\left(  ^{m}E\right)  \right)
_{m=0}^{\infty}$  and $\left(\mathcal{P}_{\sigma(p)}\left(^{m}E\right)\right)
_{m=0}^{\infty}$ are $\pi_1$-holomorphy types, respectively. Hence,  for $k\in(1,+\infty]$ and $A\in\left[  0,+\infty\right)  $, every nontrivial convolution operator $$L\colon Exp_{\widetilde{N},\left(  \left(
r,q\right)  ;\left(  s,p\right)  \right),0, A}^{k}(E)\rightarrow Exp_{\widetilde{N},\left(  \left(
r,q\right)  ;\left(  s,p\right)  \right),0, A}^{k}(E)$$ or $$L\colon Exp_{\sigma(p),0, A}^{k}(E)\rightarrow Exp_{\sigma(p),0, A}^{k}(E)$$ is mixing. For details about the duality results given by the Borel transform and the theory involving the Lorentz polynomials and $\sigma(p)$-nuclear polynomials we refer to \cite{FMP, FP, Matos-Pellegrino}  and \cite{favaro3, BM}, respectively.
%

{\em Authors' addresses}: 
Faculdade de Matem\'{a}tica, Universidade Federal de Uberl\^{a}ndia,
38.400-902 - Uberl\^{a}ndia, Brazil,\newline
 e-mails: \texttt{vvfavaro@gmail.com}

\qquad\texttt{marquesjatoba@ufu.br}

\end{document}